\documentclass[reqno]{amsart}
\usepackage{amsmath, mathtools, mathrsfs}
\usepackage{amsfonts}
\usepackage{amssymb}
\usepackage{graphicx,color}
\usepackage[square,numbers]{natbib}
\usepackage{verbatim}
\newtheorem{theorem}{Theorem}
\newtheorem{claim}[theorem]{Claim}
\newtheorem{corollary}[theorem]{Corollary}
\newtheorem{lemma}[theorem]{Lemma}
\newtheorem{proposition}[theorem]{Proposition}
\newtheorem{algorithm}[theorem]{Algorithm}

\theoremstyle{definition}
\newtheorem{definition}[theorem]{Definition}
\newtheorem{assumptions}[theorem]{Assumptions}
\newtheorem{notation}[theorem]{Notation}

\theoremstyle{remark}
\newtheorem{remark}[theorem]{Remark}

\numberwithin{theorem}{section}
\numberwithin{equation}{section}

\def\XXint#1#2#3{{\setbox0=\hbox{$#1{#2#3}{\int}$}
     \vcenter{\hbox{$#2#3$}}\kern-.5\wd0}}

\newcommand{\bbH}{\mathbb{H}}

\newcommand{\bbE}{\mathbb{E}}
\newcommand{\bbG}{\mathbb{G}}
\newcommand{\bbF}{\mathbb{F}}

\newcommand{\bbN}{\mathbb{N}}

\newcommand{\bbR}{\mathbb{R}}
\newcommand{\bbQ}{\mathbb{Q}}

\begin{document}
\title[UDS in Carnot Groups of Arbitrarily High Step]{Universal Differentiability Sets in Carnot Groups of Arbitrarily High Step}

\author[Andrea Pinamonti]{Andrea Pinamonti}
\address[Andrea Pinamonti]{Department of Mathematics, University of Trento, Via Sommarive 14, 38050 Povo (Trento), Italy}
\email[Andrea Pinamonti]{Andrea.Pinamonti@unitn.it}

\author[Gareth Speight]{Gareth Speight}
\address[Gareth Speight]{Department of Mathematical Sciences, University of Cincinnati, 2815 Commons Way, Cincinnati, OH 45221, United States}
\email[Gareth Speight]{Gareth.Speight@uc.edu}


\date{\today}

\begin{abstract}
We show that any Carnot group $\bbG$ with sufficiently many \emph{deformable directions} contains a measure zero set $N$ such that every Lipschitz map $f\colon \bbG\to \bbR$ is differentiable at some point of $N$. We also prove that model filiform groups satisfy this condition, extending some previous results to a class of Carnot groups of arbitrarily high step.
Essential to our work is the question of whether the existence of an (almost) maximal directional derivative $Ef(x)$ in a Carnot group implies the differentiability of a Lipschitz map $f$ at $x$. We show that such an implication is valid in model Filiform groups for directions that are outside a one-dimensional subspace of horizontal directions. Conversely, we show that this implication fails for every horizontal direction in the free Carnot group of step three and rank two.
\end{abstract}

\maketitle

\section{Introduction}\label{intro}

Rademacher's theorem asserts that every Lipschitz map $f\colon \bbR^{n}\to \bbR^{m}$ is differentiable almost everywhere with respect to the Lebesgue measure. This important result has been extended to many other spaces and measures \cite{AM14, Che99, LPT13, Pan89}. It is also interesting to consider whether Rademacher's theorem admits a converse: given a Lebesgue null set $N\subset \bbR^{n}$, does there exist a Lipschitz map $f\colon \bbR^{n}\to \bbR^{m}$ which is differentiable at no point of $N$? The answer to this question is yes if and only if $n\leq m$ and combines the work of several authors \cite{Zah46, Pre90, PS15, ACP10, CJ15}. In the case where $n>m=1$, the results in \cite{DM11, DM12, DM14} provide a stronger result: there is a compact set of Hausdorff dimension one in $\bbR^{n}$ which contains some point of differentiability of any Lipschitz map $f\colon \bbR^{n}\to \bbR$. Such a set may even be chosen with upper Minkowski dimension one \cite{DM14}. Sets containing a point of differentiability for any real-valued Lipschitz map are called \emph{universal differentiability sets}. We refer the reader to \cite{PM} and the references therein for more discussion of such sets.

The present paper continues the investigation of universal differentiability sets in Carnot groups which was started in \cite{PS16, LPS17}, see also the survey \cite{PSsurvey}. 
We recall that a Carnot group (Definition \ref{Carnot}) is a simply connected Lie group whose Lie algebra $\mathfrak g$ admits a stratification, i.e. it admits a decomposition $\mathfrak g= V_1\oplus \dots \oplus V_s$  where $V_{i+1}=[V_1,V_i]$ for $i=1,\dots, s-1$. The subspace $V_1$ is called the horizontal layer while $s$ is the step of the Carnot group and to some extent indicates its complexity (Carnot groups of step one are simply Euclidean spaces). Carnot groups have a rich geometric structure adapted to the horizontal layer, including translations, dilations, Carnot-Carath\'{e}odory (CC) distance, and a Haar measure \cite{ABB, CDPT07, Gro96, Mon02}. In the last two decades Carnot groups have been studied in connection with several different areas of mathematics, such as PDE, differential geometry, control theory, geometric measure theory, mathematical finance and robotics. Their rich structure allows one to define differentiability of maps between Carnot groups (Definition \ref{pansudifferentiability}). Pansu's theorem states that every Lipschitz map is differentiable almost everywhere with respect to the Haar measure \cite{Pan89}. This is a generalization of Rademacher's theorem to Carnot groups.

In \cite{PS16}, it was shown that Heisenberg groups contain measure zero universal differentiability sets. Heisenberg groups are the most frequently studied non-Euclidean Carnot groups and have step two. In \cite{LPS17} this result was extended to give a measure zero and Hausdorff dimension one universal differentiability set in any step two Carnot group. The present paper extends these results and the associated techniques to higher step Carnot groups satisfying a precise geometric condition, namely having sufficiently many \emph{deformable directions} (Definition \ref{deform}). This is a geometric condition expressing that, roughly speaking, horizontal lines can be nicely modified to pass through nearby points, without changing too much their length or their direction. This condition applies in particular to model filiform groups (Definition \ref{filiform}), which can have arbitrarily high step despite their relatively simple Lie brackets. Model filiform groups have been previously investigated in connection with non-rigidity of Carnot groups \cite{O08}, quasiconformal mappings between Carnot groups \cite{War03, Xia15} and geometric control theory \cite{BLU07}.

Before describing more carefully the results of this paper, we briefly discuss the techniques involved in constructing universal differentiability sets. We believe these are of independent interest as they only depend on the geometry of the space involved. In \cite{Pre90, PS16, LPS17}, the key technique for constructing measure zero universal differentiability sets builds upon the idea that existence of a \emph{maximal directional derivative} for a Lipschitz map suffices for its differentiability. In Euclidean spaces, this observation takes the following form: if $f\colon \bbR^{n} \to \bbR$ is Lipschitz and $|f'(x,v)|= \mathrm{Lip}(f)$ for some direction $v\in \bbR^{n}$ with $|v|=1$, then $f$ is differentiable at $x$, see \cite{Fit84}. However, a general Lipschitz map may not have such a maximal directional derivative. In \cite{Pre90} it was shown that any Lipschitz map $f\colon \bbR^{n} \to \bbR$ admits a linear perturbation that has an \emph{almost maximal directional derivative} at some point $x$ in a direction $v$. They also show that almost maximality suffices for differentiability and the point $x$ can be chosen inside a measure zero set $N$ that is independent of $f$. Combining the two facts we have that $N$ is a universal differentiability set of measure zero.

In \cite{PS16, LPS17}, the present authors and E. Le Donne showed that if $f\colon \bbG \to \bbR$ is a Lipschitz map on a step two Carnot group and $Ef(x)$ is a maximal directional derivative (Definition \ref{maximal}), then $f$ is differentiable at $x$ in the sense of Pansu, see Definition \ref{pansudifferentiability}). Moreover, in step two Carnot groups, it was also shown that almost maximality of a directional derivative suffices for differentiability. Generalizing the Euclidean techniques one can then construct a measure zero and Hausdorff dimension one universal differentiability set. Moreover, for each horizontal direction $E$ in an arbitrary Carnot group, differentiability of the CC distance at $\exp(E)$ is equivalent to validity of the following implication: maximality of $Ef(x)$ for Lipschitz $f\colon \bbG\to \bbR$ implies differentiability of $f$ at $x$ (Proposition \ref{equivalence}). However, in the Engel group, which represents the simplest step 3 Carnot group, neither of the above properties hold. The counterexample is simply given by the horizontal direction $X_{2}$, since the CC distance fails to be differentiable at $\exp(X_{2})$. It is then clear that the geometry of the space impacts the differentiability of its Lipschitz maps.

The reason why maximality implies differentiability (Proposition \ref{equivalence}(2)) fails for the direction $X_{2}$ in the Engel group is that horizontal lines in the direction $X_{2}$ cannot be modified to pass through nearby points without increasing too much their length. If `maximality implies differentiability' fails, then so does the stronger implication `almost maximality implies differentiability'. This stronger implication depends upon the possibility of modifying horizontal lines with some controlled bounds on both their length \emph{and} their direction. This stronger modification is useful because in `almost maximality' the directional derivative is maximal only compared to directional derivatives coming from pairs of points and directions which satisfy estimates expressed using difference quotients of the Lipschitz map. A direction is deformable if suitable deformations of horizontal lines are possible. All horizontal directions in step two Carnot groups are deformable. This was proved in \cite{LPS17}, though the word deformable was not used there. In the present paper we show that in model filiform groups $\bbE_{n}$, any horizontal direction other than $\pm X_{2}$ is deformable (Theorem \ref{deformFiliform}). We also show  that $\pm X_{2}$ are deformable in $\bbE_{n}$ if and only if $n=2$ or $n=3$ (Corollary \ref{pmX2}). 

A set $N$ in a Carnot group $\bbG$ is a universal differentiability set if every Lipschitz map $f\colon \bbG \to \bbR$ is differentiable at some point of $N$ (Definition \ref{defUDSabstract}). We say that a set has CC Hausdorff dimension one if it has Hausdorff dimension one with respect to the CC metric. Our main result is the following.

\begin{theorem}\label{maintheorem}
Let $\bbG\neq \bbR$ be any Carnot group that has a ball of uniformly deformable directions (see Assumptions \ref{ass}). Then $\bbG$ contains a universal differentiability set $N\subset \bbG$ of CC Hausdorff dimension one (in particular measure zero).

In particular, all model filiform groups $\bbE_{n}$ for $n\geq 2$ contain a CC Hausdorff dimension one universal differentiability set.
\end{theorem}

A ball of uniformly deformable directions is needed in Theorem \ref{maintheorem} because one constructs the measure zero UDS using countably many horizontal curves which are dense in some sense. To prove that `almost maximality implies differentiability' (Theorem \ref{almostmaximalityimpliesdifferentiability}) one needs to approximate the almost maximal direction with a sequence of deformable ones. If we do not require that the UDS has measure zero, then only one deformable direction is needed to show that almost maximality implies differentiability.

Notice that Theorem \ref{maintheorem} applies to the Engel group $\bbE_{4}$, which was the problematic group in \cite{LPS17}. Hence one may ask whether Theorem \ref{maintheorem} holds without assuming Assumptions \ref{ass}. Our second result shows that, unless one fundamentally changes the techniques used, the class of Carnot groups must indeed be restricted.

\begin{theorem}\label{strongnondiff}
In the free Carnot group $\bbF_{2,3}$ with rank two and step three, the CC distance is not differentiable at $\exp(E)$ for any horizontal direction $E\in V_{1}$. 

Consequently, `maximality implies differentiability' fails in $\bbF_{2,3}$ for every horizontal direction.
\end{theorem}

This improves in a strong way upon \cite{LPS17}, where it was shown that maximality implies differentiability fails for one direction in the Engel group. It would be interesting to know whether $\bbF_{2,3}$ contains a measure zero UDS or instead the opposite result holds: for every null set $N\subset \bbF_{2,3}$, does there exists a Lipschitz map $f\colon \bbF_{2,3}\to \bbR$ which is differentiable at no point of $N$? At present we do not know the answer to this question. 

The notion of deformability introduced in the present paper seems to share some analogy with the property of not being an abnormal curve, see \cite{Vit} for the definition of abnormal curves. For example, it is known that in $\bbF_{2,3}$ the set of all abnormal curves coincides with the set of horizontal lines \cite{ABB}. This phenomenon could explain why `maximality implies differentiability' fails in $\bbF_{2,3}$ for every horizontal direction. A similar characterization holds in model filiform groups, which admit only one abnormal curve which is given by the $\pm X_{2}$ direction \cite{ABB}. However the picture is far from being clear. For example in $\bbF_{3,2}$, where `maximality implies differentiability' holds for every horizontal direction \cite{LPS17}, every horizontal line is an abnormal curve and viceversa, see \cite[Proposition 3.11, Theorem 3.14]{LMOPV16} and \cite{OV17}. We plan to investigate this possible relation in future works.

We now describe the structure of the paper. In Section \ref{preliminaries} we recall the necessary background on Carnot groups and differentiability. In Section \ref{CCdifferentiability} we investigate the differentiability of the CC distance. We show that, if $E$ is a deformable direction, then the CC distance is differentiable at $\exp(E)$ (Proposition \ref{deformimpliesdiff}) and that in any model filiform group $\bbE_{n}$, with $n\geq 4$, the CC distance is not differentiable at $\exp(\pm X_{2})$ (Proposition \ref{X2nogood}). 
We eventually prove Theorem \ref{strongnondiff}. In Section \ref{CurvesFiliform} we prove Lemma \ref{Xn} and Lemma \ref{Filiformcurve} which allow us to construct suitable horizontal curves in model filiform groups. These are then used to show that in every model filiform group all horizontal directions other than $\pm X_{2}$ are deformable (Theorem \ref{deformFiliform}). In Section \ref{sectiondistanceestimate} we prove an estimate for distances between piecewise linear curves with similar directions (Lemma \ref{closedirectioncloseposition}). In Section \ref{sectionUDS} we consider Carnot groups $\bbG$ that contain a ball of uniformly deformable directions with parameters. With this assumption we construct a universal differentiability set (Lemma \ref{uds}) and prove that almost maximality implies differentiability if the direction belongs to the given ball (Theorem \ref{almostmaximalityimpliesdifferentiability}). In Section \ref{sectionconstruction} we show that any Lipschitz map $f\colon \bbG\to \bbR$ admits a group linear perturbation which has an almost maximal direction derivative at some point $x$ in some horizontal direction $E$ (Proposition \ref{DoreMaleva}). Moreover, the point $x$ can be found inside a given measure zero $G_{\delta}$ set and the direction $E$ can be found close to a starting direction $E_{0}$. A proof of Theorem \ref{maintheorem} is given by combining Theorem \ref{almostmaximalityimpliesdifferentiability} and Proposition \ref{DoreMaleva}.

\smallskip

\noindent \textbf{Acknowledgement.} The authors thank the referees for very detailed comments which greatly improved the presentation of the paper.

Part of this work was done while G. Speight was visiting the University of Trento; he thanks the institution for its hospitality. This work was supported by a grant from the Simons Foundation (\#576219, G. Speight). A. P. is a member of {\em Gruppo Nazionale per l'Analisi Ma\-te\-ma\-ti\-ca, la Probabilit\`a e le loro Applicazioni} (GNAMPA) of {\em Istituto Nazionale di Alta Matematica} (INdAM).

\section{Preliminaries}\label{preliminaries}

In this section we recall concepts which will be important throughout the paper.

\subsection{Basic notions in Carnot groups}

\begin{definition}\label{Carnot}
A \emph{Carnot group} $\bbG$ of \emph{step} $s$ is a simply connected Lie group whose Lie algebra $\mathfrak{g}$ admits a decomposition as a direct sum of subspaces of the form
\[\mathfrak{g}=V_{1}\oplus V_{2}\oplus \cdots \oplus V_{s}\]
such that $V_{i}=[V_{1},V_{i-1}]$ for any $i=2, \ldots, s$, and $[V_{1},V_{s}]=0$. The subspace $V_{1}$ is called the \emph{horizontal layer} and its elements are called \emph{horizontal left invariant vector fields}. The \emph{rank} of $\bbG$ is $\dim(V_{1})$.
\end{definition}

The exponential mapping $\exp\colon \mathfrak{g}\to \bbG$ is a diffeomorphism. Given a basis $X_{1},\ldots, X_{n}$ of $\mathfrak{g}$ adapted to the stratification, any $x\in \bbG$ can be written in a unique way as
\[x=\exp(x_{1}X_{1}+\ldots +x_{n}X_{n}).\]
We identify $x$ with $(x_{1},\ldots, x_{n})\in \bbR^{n}$ and hence $\bbG$ with $\bbR^{n}$. This is known as \emph{exponential coordinates of the first kind}. To compute the group law in these coordinates, one uses the equality
\[\exp(X)\exp(Y)=\exp(X\diamond Y)\quad \mbox{ for all } X,Y\in\mathfrak{g}.\]
Here $\diamond$ is defined by the Baker-Campbell-Hausdorff (BCH) formula
\begin{align}\label{BCH}
X\diamond Y= X+Y+\frac{1}{2}[X,Y]+\frac{1}{12}([X,[X,Y]]+[Y,[Y,X]]) + \ldots,
\end{align}
where higher order terms are nested commutators of $X$ and $Y$ \cite{Var}, see e.g. \cite[Theorem 2.2.13]{BLU07}.

Unless otherwise stated, $\bbG$ will be a Carnot group of step $s$ and rank $r$ with $\dim(\mathfrak{g})=n$ which is represented in exponential coordinates of the first kind.

We say that a curve $\gamma \colon [a,b]\to \bbG$ is absolutely continuous if it is absolutely continuous as a curve into $\bbR^{n}$. Fix a basis $X_{1}, \ldots, X_{r}$ of $V_{1}$ and an inner product norm $\omega$ on $V_{1}$ making the chosen basis orthonormal.

\begin{definition}\label{horizontalcurve}
An absolutely continuous curve $\gamma\colon [a,b]\to \bbG$ is \emph{horizontal} if there exist $u_{1}, \ldots, u_{r}\in L^{1}[a,b]$ such that
\[\gamma'(t)=\sum_{j=1}^{r}u_{j}(t)X_{j}(\gamma(t)) \quad \mbox{for almost every }t\in [a,b].\]
The \emph{length} of such a curve is $L_{\bbG}(\gamma):=\int_{a}^{b}|u|$.
\end{definition}

Since $\bbG$ is identified with $\bbR^{n}$ as a manifold, its tangent spaces are also naturally identified with $\bbR^{n}$. We say that a vector $v\in \bbR^n$ is \emph{horizontal} at $p\in \bbG$ if $v=E(p)$ for some $E\in V_1$. Thus a curve $\gamma$ is horizontal if and only if $\gamma'(t)$ is horizontal at $\gamma(t)$ for almost every $t$. All the curves of the form $t\mapsto p\exp(tV)$ for some $p\in \bbG$ and $E\in V_{1}$ are horizontal and they will be called \emph{horizontal lines}.

Chow's theorem \cite{Chow39} asserts that any two points in a Carnot group can be connected by a horizontal curve. Hence the following definition gives a metric on $\bbG$.

\begin{definition}\label{carnotdistance}
The \emph{Carnot-Carath\'{e}odory (CC) distance} between any two points $x, y\in \bbG$ is defined by
\[d(x,y):=\inf \{L_{\bbG}(\gamma)\colon \gamma \mbox{ is a horizontal curve joining } x \mbox{ and }y\}.\]
We also use the notation $d(x):=d(x,0)$ for $x\in \bbG$.
\end{definition}

Left group translations preserve lengths of horizontal curves. This implies
\[d(gx,gy)=d(x,y) \quad \mbox{for every }g,x,y \in \bbG.\]
Even though the CC distance and the Euclidean distance are not Lipschitz equivalent, they induce the same topology. Hence $\bbQ^{n}$ is dense in $\bbR^{n}$ with respect to the CC distance. The following proposition will be useful to compare the two distances \cite{NSW}, see also \cite[Corollary 5.2.10 and Proposition 5.15.1]{BLU07}.

\begin{proposition}\label{euclideanheisenberg}
Let $\bbG$ be a Carnot group of step $s$ and $K\subset \bbG$ be a compact set. Then there exists a constant $C_{\mathrm{H}} \geq 1$ depending on $K$ such that
\[ C_{\mathrm{H}}^{-1} |x-y|\leq d(x,y)\leq C_{\mathrm{H}}|x-y|^{1/s} \qquad \mbox{for all }x, y\in K.\]
\end{proposition}

We will also need the following estimate for the CC distance \cite[Lemma 2.13]{FS}. 

\begin{proposition}\label{conjugatedistance}
Let $\bbG$ be a Carnot group of step $s$. Then there is a constant $C_D\geq 1$ such that
\begin{align*}
d(x^{-1}yx) \leq C_D\,\Big(d(y)+ d(x)^{\frac{1}{s}}d(y)^{\frac{s-1}{s}}+d(x)^{\frac{s-1}{s}}d(y)^{\frac{1}{s}}\Big)\quad \mbox{for }x,y\in\bbG.
\end{align*}
\end{proposition}

\begin{definition}\label{dilations}
For any $\lambda>0$, we define the \emph{dilation} $\delta_{\lambda}\colon \bbG \to \bbG$ in coordinates by
\[\delta_{\lambda}(x_{1}, \ldots, x_{n})=(\lambda^{\alpha_{1}}x_{1},\ldots, \lambda^{\alpha_{n}}x_{n})\]
where $\alpha_{i}\in \bbN$ is the homogeneity of the variable $x_{i}$, which is defined by
\[\alpha_{j}=i \qquad \mbox{whenever} \qquad h_{i-1}+1<j\leq h_{i},\]
where $h_{i}:=\dim(V_{1}) + \ldots \dim(V_{i})$ for $i\geq 1$ and $h_{0}:=0$. For our purposes, it will be enough to know that $\alpha_{1}=\cdots=\alpha_{r}=1$, where $r=\dim(V_{1})$.
\end{definition}

Dilations are group homomorphisms of $\bbG$ and they satisfy
\[d(\delta_{\lambda}(x),\delta_{\lambda}(y))=\lambda d(x,y) \quad \mbox{for every }x, y \in \bbG \mbox{ and }\lambda>0.\]
We will also use the fact that $\delta_{\lambda}(\exp(E))=\exp (\lambda E)$ for every $\lambda >0$ and $E\in V_{1}$.

Carnot groups have a Haar measure which is unique up to scalars. When $\bbG$ is represented in first exponential coordinates as $\bbR^{n}$, the Haar measure is simply the Lebesgue measure $\mathcal{L}^{n}$, which satisfies
\[\mathcal{L}^{n}(gA)=\mathcal{L}^{n}(A) \qquad \mbox{and}\qquad \mathcal{L}^{n}(\delta_{\lambda}(A))=\lambda^{Q}\mathcal{L}^{n}(A)\]
for every $g\in \bbG$, $\lambda>0$ and $A\subset \bbG$ measurable. Here $Q:=\sum_{i=1}^{s}i\dim(V_{i})$ is the \emph{homogeneous dimension} of $\bbG$, which is also the Hausdorff dimension of $\bbG$ with respect to the CC metric.

\subsection{Differentiability in Carnot groups}

\begin{definition}\label{defdirectionalderivative}
Let $f\colon \bbG \to \bbR$ be a Lipschitz function, $x\in \bbG$ and $E\in V_{1}$. The \emph{directional derivative of $f$ at $x$ in direction $E$} is defined by
\[Ef(x):=\lim_{t\to 0} \frac{f(x\exp(tE))-f(x)}{t},\]
whenever the limit exists.
\end{definition}

Pansu defined the notion of differentiability in Carnot groups and proved a Rademacher theorem for maps between general Carnot groups \cite{Pan89}. We will only be concerned with the case where the target is $\bbR$.

\begin{definition}\label{pansudifferentiability}
A function $L\colon \bbG \to \bbR$ is \emph{$\bbG$-linear} if $L(xy)=L(x)+L(y)$ and $L(\delta_{r}(x))=rL(x)$ for all $x, y\in \bbG$ and $r>0$.

Let $f\colon \bbG\to \bbR$ and $x\in \bbG$. We say that $f$ is \emph{differentiable at $x$} if there is a $\bbG$-linear map $L \colon \bbG\to \bbR$ such that
\[\lim_{y \to x} \frac{|f(y)-f(x)-L(x^{-1}y)|}{d(x,y)}=0.\]
In this case we say that $L$ is the \emph{Pansu differential} of $f$ at $x$.
\end{definition}

\begin{theorem}[Pansu]\label{pansutheorem}
Every Lipschitz function $f\colon \bbG \to \bbR$ is differentiable almost everywhere with respect to the Haar measure on $\bbG$.
\end{theorem}

Note that Theorem \ref{pansutheorem} also holds for Carnot group targets \cite{Pan89} and even for suitable infinite dimensional targets \cite{MR, MPS17}.

\begin{definition}\label{defUDSabstract}
A set $N\subset \bbG$ is called a \emph{universal differentiability set (UDS)} if every Lipschitz map $f\colon \bbG \to \bbR$ is differentiable at a point of $N$.
\end{definition}

Theorem \ref{pansutheorem} implies that every positive measure subset of $\bbG$ is a UDS \cite{Mag01}. In Euclidean space $\bbR^{n}$ for $n>1$, where the group law is simply addition and the step is $1$, measure zero UDS exist and they can be made compact and of Hausdorff and Minkowski dimension one \cite{Pre90, DM12, DM14}. All step 2 Carnot groups contain a measure zero UDS of Hausdorff dimension one with respect to the CC metric \cite{PS16, LPS17}. Note that the Hausdorff dimension of any UDS must be at least one \cite{LPS17}.

Define the horizontal projection $p\colon \bbG\to \bbR^{r}$ by $p(x):=(x_{1},\ldots, x_{r})$. We now recall some relevant results from \cite{LPS17}.

\begin{lemma}\label{distanceinequality}
Let $u=\exp(E)$ for some $E\in V_{1}$. Then
\[d(uz) \geq d(u)+ \langle p(z), p(u)/d(u)\rangle \qquad \mbox{for any }z\in \bbG.\]
Moreover, if the CC distance $d\colon \bbG\to \bbR$ is differentiable at $u=\exp(E)$, then its Pansu differential at $u$ takes the form $z\mapsto \langle p(z), p(u)/d(u)\rangle$.
\end{lemma}

Recall that $\omega$ is an inner product norm on $V_{1}$ making the basis $X_{1}, \ldots, X_{r}$ of $V_{1}$ orthonormal. We have the following connection between directional derivatives and the Lipschitz constant of a Lipschitz map.

\begin{lemma}\label{lipismaximal}
Let $f\colon \bbG \to \bbR$ be a Lipschitz map. Then
\[\mathrm{Lip}(f)=\sup\{|Ef(x)| \colon x\in \bbG, \, E\in V_{1}, \, \omega(E)=1, \, Ef(x) \mbox{ exists}\}.\]
\end{lemma}

This justifies the definition of maximal directional derivative.

\begin{definition}\label{maximal}
Let $f\colon G\to \mathbb{R}$ be Lipschitz and let $E\in V_{1}$ with $\omega(E)=1$. We say that a directional derivative $Ef(x)$ is \emph{maximal} if $|Ef(x)|=\mathrm{Lip}(f)$. 
\end{definition}

In Euclidean spaces (and Banach spaces with a differentiable norm), maximality of a directional derivative suffices for differentiability. The following proposition from \cite{LPS17} gives a condition for `maximality implies differentiability' using the differentiability of the CC distance.

\begin{proposition}\label{equivalence}
Let $E\in V_{1}$ with $\omega(E)=1$. Then the following are equivalent.
\begin{enumerate}
\item The CC distance $d$ is differentiable at $\exp(E)$.
\item The following implication holds: whenever $f\colon \bbG \to \bbR$ is Lipschitz and $Ef(x)$ is maximal for some $x\in \bbG$, then $f$ is differentiable at $x$.
\end{enumerate}
\end{proposition}

Known constructions of measure zero UDS rely upon a stronger implication, namely that the existence of an \emph{almost maximal} directional derivative suffices for differentiability. To investigate this stronger implication, deformability as defined below will be important. First recall that two horizontal curves $f_{1}\colon [a,b]\to \bbG$ and $f_{2}\colon [b,c]\to \bbG$ with $f_{1}(b)=f_{2}(b)$ can be joined to form a horizontal curve $f\colon [a,c]\to \bbG$ given by $f(t)=f_{1}(t)$ if $a\leq t\leq b$ and $f(t)=f_{2}(t)$ if $b\leq t\leq c$. Similarly one can join any finite number of horizontal curves provided the end of each curve agrees with the start of the subsequent curve.

\begin{definition}\label{deform}
We say that $E\in V_{1}$ with $\omega(E)=1$ is \emph{deformable} if there exist $C_{E}$, $N_{E}$ and a map $\Delta_{E}\colon (0,\infty) \to (0, \infty)$ such that the following condition holds.

For every $0<s<1, \eta \in (0,\infty), 0<\Delta<\Delta_{E}(\eta)$ and $u\in \bbG$ with $d(u)\leq 1$, there is a Lipschitz horizontal curve $g\colon \bbR\to \bbG$ formed by joining at most $N_{E}$ horizontal lines such that
\begin{enumerate}
\item $g(t)=\exp(tE)$ for $|t|\geq s$,
\item $g(\zeta)=\delta_{\Delta s}(u)$, where $\zeta:= \langle \delta_{\Delta s}(u),E(0)\rangle$,
\item $\mathrm{Lip}_{\bbG}(g)\leq 1+\eta \Delta$,
\item $|(p\circ g)'(t)-p(E)|\leq C_{E}\Delta$ for all but finitely many $t\in \bbR$.
\end{enumerate}
\end{definition}

\begin{remark}\label{deform2}
Consider the restriction of the curve $g$ from Definition \ref{deform} to the interval $[-s,\zeta]$. By applying left translations and reparameterizing, we obtain a curve $\varphi\colon [0,s+\zeta]\to \bbG$ with $\varphi(0)=0$, $\varphi(s+\zeta)=\exp(sE)\delta_{\Delta s}(u)$, and satisfying conditions 3 and 4 of Definition \ref{deform}.
\end{remark}

\subsection{Useful facts about exponential coordinates of the first kind}

In the first $r$ coordinates, the group operation and dilations $\delta_{\lambda}$ read as
\[p(xy)=p(x)+p(y) \quad \mbox{and} \quad p(\delta_{\lambda}(x))=\lambda p(x) \quad \mbox{for }x,y\in \bbG \mbox{ and }\lambda>0.\]

Let $e_{1}, \ldots, e_{n}$ be the standard basis vectors of $\bbR^{n}$. If $1\leq j\leq r$, the basis element $X_{j}$ of $V_{1}\subset \mathfrak{g}$ can be written as
\begin{equation}\label{vfcoordinates}
X_{j}(x)=e_{j}+\sum_{i>r}^{n}q_{i,j}(x)e_{i},
\end{equation}
where $q_{i,j}$ are homogeneous polynomials, in particular, $q_{i,j}(0)=0$. Using \eqref{vfcoordinates}, it follows that $\exp (E)=E(0)$ for any $E\in V_{1}$. Thus points $u=\exp (E)$ for some $E\in V_{1}$ are exactly those of the form $u=(u_{h},0)$ for some $u_{h}\in \bbR^{r}$ and therefore $\exp(E)=(p(\exp(E)),0)$. If $E\in V_{1}$, it follows from \eqref{vfcoordinates} that $p(E(x))$ is independent of $x\in \bbG$. Hence one can unambiguously define $p(E)\in \bbR^{r}$ for every $E\in V_{1}$. The inner product norm $\omega$ is equivalently given by $\omega(E)=|p(E)|$. 

From Definition \ref{horizontalcurve} and \eqref{vfcoordinates} we notice that $L_{\bbG}(\gamma)=L_{\bbE}(p \circ \gamma)$, where $L_{\bbE}$ is the Euclidean length of a curve in $\bbR^{r}$. This implies that
\[d(x,y)\geq |p(y)-p(x)| \quad \mbox{for all} \quad x, y\in \bbG.\]

Lemma \ref{horizontaldistances} and Lemma \ref{lipschitzhorizontal} below give useful facts about length and distance in coordinates. They can be proved exactly as in \cite[Lemma 2.8 and Lemma 2.9]{PS16}.

\begin{lemma}\label{horizontaldistances}
If $E\in V_{1}$ then the following facts hold.
\begin{enumerate}
\item $|E(0)|=\omega(E)=d(E(0))$,
\item $d(x,x\exp(tE))=t\omega(E)$ for any $x\in \bbG$ and $t\in \bbR$.
\end{enumerate}
\end{lemma}

\begin{lemma}\label{lipschitzhorizontal}
Suppose that $\gamma \colon I \to \bbG$ is a horizontal curve. Then
\[\mathrm{Lip}_{\bbG}(\gamma) =  \mathrm{Lip}_{\bbE}(p \circ \gamma).\]
\end{lemma}

The following lemma gives easy facts about the simplest $\bbG$-linear maps. It can be easily proved, e.g. as in \cite[Lemma 5.2]{PS16}.

\begin{lemma}\label{lemmascalarlip}
Suppose $E\in V_{1}$ with $\omega(E)=1$ and let $L\colon \bbG \to \bbR$ be the function defined by $L(x)=\langle x, E(0) \rangle$. Then the following facts hold.
\begin{enumerate}
\item $L$ is $\bbG$-linear and $\mathrm{Lip}_{\bbG}(L) = 1$,
\item for every $x\in\bbG$ and every $\tilde E\in V_{1}$ one has
\[\tilde{E}L(x)=L(\tilde{E}(0))=\langle p(\tilde{E}), p(E) \rangle.\]
\end{enumerate}
\end{lemma}

\subsection{Free Carnot groups and model filiform groups}

Recall that a homomorphism between Lie algebras is a linear map that preserves the Lie bracket, while isomorphisms are bijective homomorphisms. Free-nilpotent Lie algebras are then defined as follows (e.g. \cite[Definition 14.1.1]{BLU07}).

\begin{definition}\label{freeliealgebra}
Let $r\geq 2$ and $s\geq 1$ be integers. We say that $\mathcal{F}_{r,s}$ is the \emph{free-nilpotent Lie algebra} with $r$ \emph{generators} $x_1, \ldots, x_{r}$ of \emph{step} $s$ if:
\begin{enumerate}
\item $\mathcal{F}_{r,s}$ is a Lie algebra generated by elements $x_1, \ldots, x_r$,
\item $\mathcal{F}_{r,s}$ is nilpotent of step $s$ (i.e., nested Lie brackets of length $s+1$ are $0$),
\item for every Lie algebra $\mathfrak{g}$ that is nilpotent of step $s$ and for every map $\Phi\colon \{x_1, \ldots, x_r\}\to \mathfrak{g}$, there is a unique homomorphism of Lie algebras $\tilde{\Phi}\colon \mathcal{F}_{r,s} \to \mathfrak{g}$ that extends $\Phi$.
\end{enumerate}
\end{definition}

We next define free Carnot groups, e.g. \cite[Definition 14.1.3]{BLU07}.

\begin{definition}\label{freecarnotgroup}
The \emph{free Carnot group} with rank $r$ and step $s$ is the Carnot group whose Lie algebra is isomorphic to the free-nilpotent Lie algebra $\mathcal{F}_{r, s}$. We denote it by $\bbF_{r,s}$.
\end{definition}

By saying that two Carnot groups are isomorphic we simply mean that they are isomorphic as Lie groups, with an isomorphism that preserves the stratification. Since Carnot groups are simply connected Lie groups, any homomorphism $\phi$ between their Lie algebras lifts to a Lie group homomorphism $F$ between the Carnot groups satisfying $dF=\phi$.

Intuitively, model filiform groups are the Carnot groups with the simplest Lie brackets possible while still having arbitrarily large step. The formal definition is as follows.

\begin{definition}\label{filiform}
Let $n\geq 2$. The \emph{model filiform group of step $n-1$} is the Carnot group $\mathbb{E}_{n}$ whose Lie algebra $\mathcal{E}_{n}$ admits a basis $X_{1}, \ldots, X_{n}$ for which the only non-vanishing bracket relations are given by $[X_{i},X_{1}]=X_{i+1}$ for $1<i<n$.

The stratification of $\mathcal{E}_{n}$ is $\mathcal{E}_{n}=V_{1}\oplus \cdots \oplus V_{n-1}$ with $V_{1}=\mathrm{Span}\{X_{1}, X_{2}\}$ and $V_{i}=\mathrm{Span}\{X_{i-1}\}$ for $1<i<n$.
\end{definition}

Throughout the paper, we will view the model Filiform group $\bbE_{n}$ in first exponential coordinates as $\bbR^{n}$ with group operation obtained from the Lie brackets by the BCH formula \eqref{BCH}.

\section{Differentiability of the CC distance in Carnot groups}\label{CCdifferentiability}

In this section we investigate the differentiability of the CC distance at endpoints of horizontal vectors. By Proposition \ref{equivalence}, this is equivalent to the implication `maximality implies differentiability'.

\subsection{Deformability implies differentiability of the CC distance}

We first observe that if $E$ is a deformable direction then the CC distance is differentiable at $\exp(E)$. At present we do not know whether the converse holds.

\begin{proposition}\label{deformimpliesdiff}
Suppose $\bbG$ is a Carnot group and let $E\in V_{1}$ with $\omega(E)=1$ be deformable. Then the CC distance is differentiable at $\exp(E)$.
\end{proposition}

\begin{proof}
First notice that Lemma \ref{distanceinequality} gives $d(\exp(E)z)\geq 1 +\langle \exp(E),z\rangle$ for any $z\in \mathbb{G}$. Hence it suffices to derive a suitable upper bound for $d(\exp(E)z)$. 

Let $\eta>0$ and $\Delta_{E}\colon (0,\infty)\to (0,\infty)$ be as in Definition \ref{deform}. Suppose $z\in \bbG$ satisfies $d(z)< \min(\Delta_{E}(\eta)/2, \ 1)$ and let $s=1/2$. Then we may choose $u\in \bbG$ with $d(u)=1$ and $0<\Delta<\Delta_{E}(\eta)$ so that $z=\delta_{\Delta s}(u)$. By applying Definition \ref{deform} with this choice of $\eta, s, u, \Delta$ we find a Lipschitz horizontal curve $g\colon \bbR \to \bbG$ satisfying $\mathrm{Lip}(g)\leq 1+\eta\Delta$, $g(-1)=\exp(-E)$ and $g(\zeta)=z$, where $\zeta=\langle z, E(0)\rangle$. Since $|\zeta|\leq 1$, we now estimate as follows:
\begin{align*}
d(\exp(E)z)=d(\exp(-E),z)&=d(g(-1),g(\zeta))\\
&\leq (1+\eta\Delta)|1+\zeta|\\
&\leq 1+\langle z, E(0)\rangle + 4\eta d(z)\\
&\leq 1+\langle z, E(0)\rangle + o(d(z)).
\end{align*}
Here $o(d(z))/d(z)\to 0$ as $z\to 0$, which follows because $2\eta\Delta/d(z)=2\eta\Delta/(\Delta s)=4\eta$ and $\eta$ can be made arbitrarily small by making $d(z)$ sufficiently small.
\end{proof}

\subsection{A strong example of non-differentiability of the CC distance}

The CC distance in the Engel group (which is the model filiform group $\bbE_{4}$) is not differentiable at $\exp(X_{2})$ \cite{LPS17}. In other words, the implication `maximality implies differentiability' fails for the direction $X_{2}$ in the Engel group. We now derive some consequences of this result for other Carnot groups.

Fix two Carnot groups $\bbG$ and $\bbH$ of rank $r$ which have horizontal layers $V$ and $W$ with the following property. There exist bases $\mathbf{X}=(X_{1}, \ldots, X_{r})$ and $\mathbf{Y}=(Y_{1}, \ldots, Y_{r})$ of $V$ and $W$ respectively, together with a Lie group homomorphism $F\colon \bbG \to \bbH$ such that $F_{*}(X_{i})=Y_{i}$ for $1\leq i\leq r$.

Equip $\bbG$ and $\bbH$ with the CC metrics $d_{\bbG}$ and $d_{\bbH}$ induced by the bases $\mathbf{X}$ and $\mathbf{Y}$ respectively. We view both $\bbG$ and $\bbH$ in exponential coordinates of the first kind and let $p_{\bbG}\colon \bbG\to \bbR^{r}$ and $p_{\bbH}\colon \bbH \to \bbR^{r}$ denote the horizontal projections. For any $u\in \bbG$ we have $p_{\bbH}(F(u))=p_{\bbG}(u)$. Also if $u=(u_{h},0)\in \bbG$ for some $u_{h}\in \bbR^{r}$, then $F(u)=(u_{h},0)\in \bbH$.

The following proposition was proven in \cite{LPS17}.

\begin{proposition}\label{quotientdiffCC}
Suppose the CC distance in $\bbG$ is differentiable at $\exp(E)$ for some $E\in V_{1}$. Then the CC distance in $\bbH$ is differentiable at $\exp(F_{*}(E))$.
\end{proposition}

Our first result about non-differentiability of the CC distance is the following.

\begin{proposition}\label{X2nogood}
In any model filiform group $\bbE_{n}$ with $n\geq 4$, the CC distance is not differentiable at $\exp(X_{2})$.
\end{proposition}

\begin{proof}
Let $X_{1},\ldots, X_{n}$ be a basis of the Lie algebra $\mathcal{E}_{n}$ such that the only non-vanishing bracket relations are given by $[X_{i}, X_{1}]=X_{i+1}$ for $1<i<n$. Let $Y_{1}, \ldots, Y_{4}$ denote a similar basis for the Lie algebra $\mathcal{E}_{4}$. Define a linear map $\Phi\colon \mathcal{E}_{n}\to \mathcal{E}_{4}$ by $\Phi(X_{i})=Y_{i}$ for $1\leq i\leq 4$ and $\Phi(X_{i})=0$ for $i>4$. It is easy to see that $\Phi$ is a Lie algebra homomorphism, hence lifts to a Lie group homomorphism $F\colon \bbE_{n}\to \bbE_{4}$ satisfying $F_{*}(X_{i})=Y_{i}$ for $1\leq i\leq 2$. By \cite[Theorem 4.2]{LPS17}, the CC distance in $\bbE_{4}$ is not differentiable at $\exp(Y_{2})$. Hence, by Proposition \ref{quotientdiffCC}, the CC distance in $\bbE_{n}$ is not differentiable at $\exp(X_{2})$.
\end{proof}

Recall that $\bbE_{2}$ is just $\bbR^{2}$ and $\bbE_{3}$ is a Carnot group of step 2. Combining the results of \cite{LPS17} with Proposition \ref{deformimpliesdiff} and Proposition \ref{X2nogood} for $n\geq 4$, we obtain the following corollary.

\begin{corollary}\label{pmX2}
In the model filiform group $\bbE_{n}$, the directions $\pm X_{2}$ are deformable for $n=2$ and $n=3$. They are not deformable for $n\geq 4$.
\end{corollary}

Our second result addressing non-differentiability of the CC distance gives an example of a Carnot group where the CC distance fails to be differentiable at the endpoint of every horizontal vector. This is Theorem \ref{strongnondiff}.

\begin{proof}[Proof of Theorem \ref{strongnondiff}]
Fix an orthonormal basis $X_{1}, X_{2}$ of the horizontal layer of $\bbF_{2,3}$. It suffices to show that the CC distance in $\bbF_{2,3}$ is not differentiable at $\exp(E)$ whenever $E=aX_{1}+bX_{2}$ with $a^2+b^2=1$. Let $Y_{1}, Y_{2}$ be a basis of the horizontal layer in the Engel group $\bbE_{4}$, where the CC distance defined using $Y_{1}, Y_{2}$ is not differentiable at $\exp(Y_{2})$. 

Define $W_{1}=bY_{1}+aY_{2}$, $W_{2}=-aY_{1}+bY_{2}$. Notice $W_{1}, W_{2}$ are orthonormal with respect to the inner product induced by $Y_{1}, Y_{2}$. Hence the CC distance in $\bbE_{4}$ obtained from $W_{1}, W_{2}$ is the same as the CC distance obtained from $Y_{1}, Y_{2}$. Now define a linear map $\Phi$ from the horizontal layer of $\bbF_{2,3}$ to the horizontal layer of $\bbE_{4}$ by $\Phi(X_{i})=W_{i}$ for $i=1,2$. Using the definition of free Lie algebra and lifting, we obtain a Lie group homomorphism $F\colon \bbF_{2,3}\to \bbE_{4}$ such that $F_{*}(X_{i})=W_{i}$ for $i=1,2$. Since $a^2+b^2=1$, we have:
\[ F_{*}(aX_{1}+bX_{2})=aW_{1}+bW_{2}=Y_{2}.\]
Since the CC distance in $\bbE_{4}$ is not differentiable at $\exp(Y_{2})$, we deduce using Proposition \ref{quotientdiffCC} that the CC distance in $\bbF_{2,3}$ cannot be differentiable at $aX_{1}+bX_{2}$.
\end{proof}

\section{Deformability in model filiform groups}\label{CurvesFiliform}

In this section we work in $\bbE_{n}$ for some $n\geq 3$. Our goal is to prove that every horizontal direction in $\bbE_{n}$ except possibly $\pm X_{2}$ is deformable. For simplicity of notation, in this section we identify the model filiform group $\bbE_{n}$ with its Lie algebra $\mathcal{E}_{n}$. Hence for $E\in \mathcal{E}_{n}$ we will simply write $E$ instead of $\exp(E)$.

\subsection{Construction of horizontal curves in filiform groups}

We start by proving two lemmas that show how a horizontal line can be perturbed to reach a nearby point. The first lemma shows how to reach a point whose $n$'th coordinate is known, with small errors in the other vertical coordinates. Given $A\in \bbR$, we will use the notation
\begin{equation}\label{EE'}
E:=X_1+A X_2 \qquad \mbox{and} \qquad E':=AX_1-X_2.
\end{equation}
Note that $E$ and $E'$ are orthogonal with respect to $\omega$, for any $A\in \bbR$.

\begin{lemma}\label{Xn}
For all $A\in\bbR$ there exist polynomials $P_{i}(x)$ for $3\leq i \leq n$ depending on $n$ such that
\begin{itemize}
\item each $P_{i}(x)$ is divisible by $x^2$
\item the coefficients of each $P_{i}(x)$ are polynomials in $A$
\end{itemize}
and the following holds. 

Then, for every $\eta \in \bbR$ there exist $\eta_{i}\in \{\pm \eta\}$ for $1\leq i\leq 2^{n-2}$ such that
\[\prod_{i=1}^{2^{n-2}} \frac{1}{2^{n-2}} (E+\eta_i E')= E+C_{n}(A^2+1)\eta X_n +\sum_{i=3}^{n} P_i(\eta) X_i,\]
where $C_{n}\neq 0$ is a constant depending on $n$ and the sign of each $\eta_{i}$ depends on $i$ but not on $\eta$ or $A$.
\end{lemma}

\begin{proof}
We define products $p_{k}(\eta)$ inductively by
\begin{equation}\label{p1}
p_{1}(\eta):=(E+\eta E')(E-\eta E')
\end{equation}
and
\begin{equation}\label{pi}
\qquad p_{k+1}(\eta):=p_{k}(\eta)p_{k}(-\eta)
\end{equation}
for all $k\geq 1$. Choose $\eta_i\in \{\pm \eta\}$ for $1\leq i\leq 2^{n-2}$ such that
\[p_{n-2}(\eta) = \prod_{i=1}^{2^{n-2}} (E+\eta_i E').\]
To establish the lemma, it suffices to prove there exist $C=C_{n}\neq 0$ and polynomials $P_{i}(x)$ as in the statement such that
\begin{equation}\label{p_{n-2}}
p_{n-2}(\eta)=2^{n-2} E + 2^{n-2}C(A^2+1)\eta X_{n}+\sum_{i=3}^{n} P_i(\eta) X_i.
\end{equation}
To do so we prove by induction that, for every $1\leq k\leq n-2$, $p_{k}$ has the form
\begin{equation}\label{pkform}
p_{k}(\eta)=2^{k}E+\sum_{i=4}^{n} S_i(\eta) X_i + \eta \sum_{i=k+2}^n\lambda_i X_{i},
\end{equation}
where $\lambda_i$ are constants with $\lambda_{k+2}\neq 0$ and $S_i(x)$ are polynomials divisible by $x^2$.

To begin proving \eqref{pkform}, notice that Definition \ref{filiform}, \eqref{BCH}, \eqref{EE'} and \eqref{p1} give
\begin{equation}\label{p1form}
p_1(\eta)=2E-(A^2+1)\eta X_3 + \sum_{i=4}^{n} Q_i(\eta) X_i,
\end{equation}
where the polynomials $Q_i(x)$ are divisible by $x$. Hence \eqref{pkform} holds for $k=1$.

Next we suppose that $p_{k}$ satisfies \eqref{pkform} for some $1\leq k \leq n-3$; we will show that $p_{k+1}$ has the desired form too. Firstly, the BCH formula \eqref{BCH} gives
\begin{align}\label{pkproduct}
p_{k+1}(\eta)&=p_{k}(\eta)p_{k}(-\eta)\nonumber \\
&=p_k(\eta)+p_k(-\eta)+\frac{1}{2}[p_{k}(\eta),p_{k}(-\eta)]\nonumber \\
&\qquad \qquad+\mbox{brackets of length $\geq 3$}.
\end{align}
A simple computation using \eqref{pkform} yields
\begin{equation}\label{pksum}
p_k(\eta)+p_k(-\eta)=2^{k+1} E+ \sum_{i=4}^{n} (S_i(\eta)+S_i(-\eta)) X_i
\end{equation}
and
\begin{equation}\label{pkbracket}
\frac{1}{2}[p_k(\eta),p_k(-\eta)]=2^{k}\eta \sum_{i=k+2}^{n-1} \lambda_i X_{i+1}+2^{k-1}\sum_{i=4}^{n-1} (S_i(\eta) -S_i(-\eta)) X_{i+1}.
\end{equation}
Note that $(S_i(\eta)-S_i(-\eta))$ is divisible by $\eta^2$ because each polynomial $S_{i}(x)$ is divisible by $x^2$. The coefficient of $X_{k+3}$ in the first term on the right hand side of \eqref{pkbracket} is $2^{k} \lambda_{k+2} \eta$, with $\lambda_{k+2}\neq 0$ coming from the induction hypothesis. The other terms in \eqref{pkproduct} are linear combinations of nested commutators of length greater or equal than three, i.e.
\[
[Z_1,[Z_2,\ldots,[Z_{M-1},Z_M]]\ldots]
\]
where $Z_i\in\{p_k(\eta),p_k(-\eta)\}$ and $M\geq 3$. By the definition of $Z_i$ we get that each of the previous commutators is a constant multiple of 
\begin{align*}
&[X_1,[X_1,\ldots,[Z_{M-1},Z_M]]\ldots]\\
&\qquad =2^{k+1}\eta \sum_{i=k+2}^{n+1-M}\lambda_i X_{i+M-1}+2^k\sum_{i=4}^{n+1-M} (S_i(\eta) -S_i(-\eta)) X_{i+M-1}.
\end{align*}
The leading term in the first sum above is a multiple of $X_{j}$ with $j\geq k+4$, while the second sum consists of terms with coefficients divisible by $\eta^2$. Combining this with \eqref{pksum} and \eqref{pkbracket} shows that $p_{k+1}$ satisfies \eqref{pkform} with $k$ replaced by $k+1$. 

It follows by induction that $p_{k}$ has the desired form for every $k$. Evaluating \eqref{pkform} at $k=n-2$ gives \eqref{p_{n-2}}. This proves the lemma.
\end{proof}

We now use the Lemma \ref{Xn} and an induction argument on the dimension of the filiform group to show how a horizontal line can be perturbed to reach a nearby point, without changing too much its length or direction.

\begin{lemma}\label{Filiformcurve}
For every $A_{0}\in \bbR$, there are numbers
\[\varepsilon=\varepsilon(A_{0},n)>0, \qquad K=K(A_{0},n)>0, \qquad N=N(n)\in \bbN\]
so that the following holds for $A\in (A_{0}-\varepsilon, A_{0}+\varepsilon)$ and $a_{2}, \ldots, a_{n} \in (-\varepsilon, \varepsilon)$.

There exist $\theta_{1}, \ldots, \theta_{N} \in \bbR$ which depend smoothly on $a_{2}, \ldots, a_{n}$ with $|\theta_{i}|\leq K\max_{j\geq 2}|a_{j}|$ for $1\leq i\leq N$ and
\begin{equation}\label{Filiformcurveeq}
\prod_{i=1}^{N} \frac{1}{N} (E+\theta_{i}E') = E+a_{2}E'+a_{3}X_{3}+\cdots+a_{n}X_{n}.
 \end{equation}
\end{lemma}

\begin{proof}
We prove the lemma by induction on $n$. If $n=2$ the result is clear as $\bbE_{2}$ is simply Euclidean space $\bbR^{2}$ as a Carnot group. Suppose the statement holds in $\bbE_{n-1}$; we will show that it also holds for $\bbE_{n}$. 

By the induction hypothesis, there exist
\[\varepsilon=\varepsilon(A_{0},n-1)>0, \qquad K=K(A_{0},n-1)>0, \qquad N=N(n-1) \in \bbN\]
such that the following holds in $\bbE_{n-1}$. For any choice of $A\in (A_{0}-\varepsilon, A_{0}+\varepsilon)$ and $a_{2}, \ldots, a_{n-1} \in (-\varepsilon, \varepsilon)$, there exist $\theta_{1}, \ldots, \theta_{N}$ which depend smoothly on $a_2, \ldots, a_{n-1}$ with $|\theta_{i}|\leq K\max_{j\geq 2}|a_{j}|$ for $1\leq i\leq N$ and
\begin{equation}\label{hypothesis}
\prod_{i=1}^{N} \frac{1}{N} (E+\theta_{i}E') = E+a_{2}E'+a_{3}X_{3}+\cdots+a_{n-1}X_{n-1} \qquad \mbox{in }\bbE_{n-1}.
\end{equation}
We now lift the above equation to $\bbE_{n}$. In other words, we consider $E$ and $E'$ as elements of the Lie algebra of $\bbE_{n}$ in the natural way. All calculations when computing the product using the BCH formula \eqref{BCH} remain the same, except for $[X_{n-1},X_{1}]$ which will be equal to $X_{n}$ rather than $0$. An easy calculation shows
\begin{equation}\label{secondterm}
[E+\theta_{1}E', E+\theta_{2}E']=(1+A^2)(\theta_{2}-\theta_{1})X_{3}.
\end{equation}
From now on we work in $\bbE_{n}$. By the BCH formula \eqref{BCH}, \eqref{hypothesis} and \eqref{secondterm}:
\begin{align*}
p_{1}&:=\prod_{i=1}^{N} \frac{1}{N} (E+\theta_{i}E') \\
&= E+a_{2}E'+a_{3}X_{3}+\cdots+a_{n-1}X_{n-1} + L(\theta_{1}, \ldots, \theta_{N})X_{n}\\
&= (1+Aa_{2})X_{1}+(A-a_{2})X_{2}+a_{3}X_{3}+\cdots+a_{n-1}X_{n-1} + L(\theta_{1}, \ldots, \theta_{N})X_{n},
\end{align*}
where $L(\theta_{1}, \ldots, \theta_{N})$ is a polynomial in $\theta_{1}, \ldots, \theta_{N}$ (with coefficients depending on $A$) with no constant term due to \eqref{secondterm}. By using Lemma \ref{EE'}, for any $\eta \in \bbR$ we can choose $\eta_{i}\in \{\pm \eta\}$ for $1\leq i\leq 2^{n-2}$ with sign depending on $i$ but not $\eta$ such that
\begin{align*}
p_{2}&:=\prod_{i=1}^{2^{n-2}} \frac{1}{2^{n-2}} (E+\eta_i E')\\
&= E+C(A^2+1)\eta X_n +\sum_{i=3}^{n} P_i(\eta) X_i\\
&= X_{1}+AX_{2} +C(A^2+1)\eta X_n +\sum_{i=3}^{n} P_i(\eta) X_i,
\end{align*}
where $P_{i}(x)$ are polynomials divisible by $x^2$ and $C=C_{n}\neq 0$. 

We now analyze $p_{1}p_{2}$. First notice that
\begin{align*}
p_{1}+p_{2}&= (2+Aa_{2})X_{1}+(2A-a_{2})X_{2}+a_{3}X_{3}+\ldots + a_{n-1}X_{n-1}\\
&\qquad \qquad +(L(\theta_{1}, \ldots, \theta_{N})+C(A^2+1)\eta)X_{n} + \sum_{i=3}^{n}P_{i}(\eta)X_{i}
\end{align*}
and
\begin{align*}
[p_{1},p_{2}]&=-(A^2+1)a_{2}X_{3}+a_{3}X_{4}+\ldots +a_{n-1}X_{n}\\
& \qquad \qquad -\sum_{i=4}^{n} (1+Aa_{2})P_{i-1}(\eta)X_{i}.
\end{align*}
By using the BCH formula \eqref{BCH}, the coefficients of $E, E', X_{3}, \ldots, X_{n}$ in $p_{1}p_{2}$ are given by
\begin{align*}
&E \qquad &2\\
&E' &a_{2}\\
&X_{3} &a_{3}+F_{3}(a_{2})+Q_{3}(\eta)\\
&\ldots &\ldots\\
&X_{i} &a_{i}+F_{i}(a_{2}, \ldots, a_{i-1}) + Q_{i}(\eta, a_{2})\\
&\ldots &\ldots\\
&X_{n} &L(\theta_{1}, \ldots, \theta_{N})+C(A^2+1)\eta + F_{n}(a_{2}, \ldots, a_{n-1}) + Q_{n}(\eta, a_{2}),
\end{align*}
where, for $i=3,\ldots, n$,
\begin{itemize}
\item $F_{i}(a_{2}, \ldots, a_{i-1})$ is a polynomial with no constant term whose coefficients depend smoothly on $A$,
\item $Q_{i}(\eta, a_{2})$ is a polynomial divisible by $\eta^{2}$ whose coefficients depend smoothly on $A$.
\end{itemize}
Define $\Phi_{A}\colon (-\varepsilon,\varepsilon)^{n} \to \bbR^{n-1}$ to be the function of $a_{2}, \ldots, a_{n-1}, \eta$ whose coordinates are given by the coefficients of $E', X_{3}, X_{4}, \ldots, X_{n}$ in the above table. Recall that $\varepsilon=\varepsilon(A_{0}, n-1)$ was chosen using the induction hypothesis, which implies $\theta_{1}, \ldots, \theta_{N}$ depend smoothly on $a_{2}, \ldots, a_{n-1}$ whenever $(a_{2}, \ldots, a_{n-1}, \eta) \in (-\varepsilon,\varepsilon)^{n}$. Notice that $\Phi_{A}(0)=0$ and the equality
\begin{equation}\label{tosolve}
\delta_{1/2}(p_{1}p_{2})= E+b_{2}E'+b_{3}X_{3}+\ldots +b_{n}X_{n}
\end{equation}
is equivalent to
\begin{equation}\label{tosolve2}
\Phi_{A}(a_{2}, \ldots, a_{n-1}, \eta)= (2b_{2}, 2^2 b_{3}, \ldots, 2^{i}b_{i}, \ldots, 2^{n}b_{n}).
\end{equation}

\begin{claim}\label{claimIFT}
There exists $\tilde{\varepsilon}=\tilde{\varepsilon}(A_{0}, n)>0$ and $\tilde{K}=\tilde{K}(A_{0},n)>0$ such that for all $A\in (A_{0}-\tilde{\varepsilon}, A_{0}+\tilde{\varepsilon})$ and $b_{2}, \ldots, b_{n} \in (-\tilde{\varepsilon}, \tilde{\varepsilon})$, the equation \eqref{tosolve2} can be solved for $a_{2}, \ldots, a_{n-1}, \eta$. One can choose the solutions so that:
\begin{enumerate}
\item $a_{2}, \ldots, a_{n-1}, \eta$ depend smoothly on $b_{2}, \ldots, b_{n}$,
\item $|a_{i}|, |\eta| \leq \tilde{K}\max_{j\geq 1} |b_{j}|$.
\end{enumerate}
\end{claim}

\begin{proof}[Proof of Claim]
Notice first that $(\partial Q_{i}/\partial \eta)(0)=0$ for each $i$, since $Q_{i}$ is divisible by $\eta^{2}$. Hence $T_{A}:=\Phi_{A}'(0)$ is a lower triangular matrix with determinant $C(A^2+1)\neq 0$. By the inverse function theorem, $\Phi_{A}$ is invertible with a $C^{1}$ inverse in a neighborhood of $0$. In other words for each $A$ and $n$, given $b_{2}, \ldots b_{n}$ sufficiently small, there exist $a_{2}, \ldots, a_{n-1}, \eta$ which depend smoothly on $b_{2}, \ldots, b_{n}$ such that \eqref{tosolve2} holds. We must show that one can use a uniform neighborhood for all $A\in (A_{0}-\tilde{\varepsilon}, A_{0}+\tilde{\varepsilon})$, where $\tilde{\varepsilon}>0$ is sufficiently small and depends on $A_{0}$ and $n$. To establish such a neighborhood and the desired bounds on $|a_{i}|, |\eta|$, we briefly study the proof of the inverse function theorem from \cite[9.2.4 Theorem]{Rud76}.

Define $\lambda(A)>0$ by $1/\lambda(A)=2\|T_{A}^{-1}\|_{\mathrm{op}}$. As the determinant of $T_{A}$ is $C(A^2+1)$ and the entries of the adjoint of $T_{A}$ are linear combinations of products of entries of $T_{A}$, there exists $\lambda>0$ depending on $A_{0}$ and $n$ such that $\lambda(A)>\lambda$ whenever $A\in (A_{0}-\varepsilon, A_{0}+\varepsilon)$. By the mean value theorem, there is $0<\tilde{\varepsilon}<\varepsilon$ depending on $A_{0}$ and $n$ such that $\|\Phi_{A}'(x)-T_{A}\|_{\mathrm{op}}<\lambda$ whenever $\|x\|<\tilde{\varepsilon}$ and $A\in (A_{0}-\tilde{\varepsilon}, A_{0}+\tilde{\varepsilon})=:I$. It follows from \cite{Rud76} that for $A\in I$, the restricted map $\Phi_{A}\colon {B(0,\tilde{\varepsilon})}\to \Phi_{A}(B(0,\tilde{\varepsilon}))$ is bijective with $C^{1}$ inverse. Since the entries of $(\Phi_{A}^{-1})'(x)$ are bounded for every $A\in I$ and every $x\in B(0, \tilde{\varepsilon})$, it follows that $\Phi_{A}$ is bi-Lipschitz for $A\in I$, with bi-Lipschitz constants depending on $A_{0}$. Hence $\Phi_{A}(B(0,\tilde{\varepsilon}))$ contains a ball $B(0,\tilde{\varepsilon}/\tilde{K})$, where $\tilde{K}\geq 1$ is some constant depending on $n$ and $A_{0}$. Replacing $\tilde{\varepsilon}$ by a slightly smaller constant (still depending only on $A_{0}$ and $n$) gives the desired neighborhood for the inversion.

Next, since $\Phi_{A}^{-1}$ is Lipschitz with some Lipschitz constant $\tilde{K}$ depending on $A_{0}$ and $n$ we have
\[ |\Phi_{A}^{-1}(2b_{2}, \ldots, 2^{n}b_{n})| \leq \tilde{K} |(2b_{2}, \ldots, 2^{n}b_{n})|,\]
which yields
\[ |(a_{2}, \ldots, a_{n-1}, \eta)| \leq 2^{n}\tilde{K} |(b_{2}, \ldots, b_{n})|.\]
This gives the desired bounds on $|a_{i}|$ and $|\eta|$.
\end{proof}

Since $\theta_{i}$ and $\eta_{i}$ depend smoothly on $a_{2}, \ldots, a_{n-1}$ and $\eta$, Claim \ref{claimIFT}(1) ensures that $\theta_{i}$ and $\eta_{i}$ depend smoothly on $b_{2}, \ldots, b_{n}$. Since $|\theta_{i}|\leq K\max_{j\geq 2}|a_{j}|$, Claim \ref{claimIFT}(2) implies that $|\theta_{i}|, |\eta|\leq K\tilde{K} \max_{j\geq 1} |b_{j}|$. To conclude, since \eqref{tosolve} and \eqref{tosolve2} are equivalent, it suffices to check that $\delta_{1/2}(p_{1}p_{2})$ is of the form of the left hand side of \eqref{Filiformcurveeq}. Indeed, since the BCH formula \eqref{BCH} implies $XX=2X$ for any element $X$ of the Lie algebra, we have
\begin{align*}
p_{1}p_{2} &= \prod_{i=1}^{N} \frac{1}{N} (E+\theta_{i}E') \prod_{i=1}^{2^{n-2}} \frac{1}{2^{n-2}} (E+\eta_i E')\\
&= \prod_{i=1}^{N2^{n-2}} \frac{1}{N2^{n-2}} (E+\tilde{\theta}_{i}E') \prod_{i=1}^{N2^{n-2}} \frac{1}{N2^{n-2}} (E+\tilde{\eta}_i E'),
\end{align*}
where the sequence $\tilde{\theta}_{i}$ repeats each term of $\theta_{i}$ for $2^{n-2}$ times and the sequence $\tilde{\eta}_i$ repeats each term of $\eta_{i}$ for $N$ times. Hence we can write
\begin{equation}\label{formproved}
\delta_{1/2}(p_{1}p_{2})=\prod_{i=1}^{N2^{n-1}}\frac{1}{N2^{n-1}} (E+\Theta_{i}E'),
\end{equation}
where the terms of the sequence $\Theta_{i}$ consist of the terms of $\tilde{\theta}_{i}$ followed by the terms of $\tilde{\eta}_i$. Notice that \eqref{formproved} is of the form given in \eqref{Filiformcurveeq}. This completes the proof.
\end{proof}

\subsection{Deformability in filiform groups}

We now use the horizontal curves built in Lemma \ref{Filiformcurve} to show that every horizontal direction in $\bbE_{n}$, except possibly $\pm X_{2}$, is deformable. Recall that by Corollary \ref{pmX2}, the directions $\pm X_{2}$ are deformable if and only if $n=2$ or $n=3$.

\begin{theorem}\label{deformFiliform}
For $n\geq 3$, every $E\in \mathcal{E}_{n}$ with $\omega(E)=1$, except for possibly $\pm X_{2}$, is deformable. Moreover, the parameters $C_{E}, N_{E}, \Delta_{E}$ related to the deformability of $E$ can be chosen so that any $\tilde{E}$ sufficiently close to $E$ is also deformable with the same parameters
\[ C_{\tilde{E}}=C_{E}, \qquad N_{\tilde{E}}=N_{E}, \qquad \Delta_{\tilde{E}}(\eta)=\Delta_{E}(\eta).\]
\end{theorem}

\begin{proof}
Fix $E\in \mathcal{E}_{n}$ with $\omega(E)=1$ and $E\neq \pm X_{2}$. We need to show that for some $C_{E}, N_{E}$ and $\Delta_{E}\colon (0,\infty)\to (0,\infty)$ the following holds. Given $0<s<1$, $\eta\in (0,\infty)$, $u\in \bbE_{n}$ with $d(u)\leq 1$ and $0<\Delta<\Delta_{E}(\eta)$, there is a Lipschitz horizontal curve $g\colon \bbR\to \bbE_{n}$, formed by joining $N_{E}$ horizontal lines, such that
\begin{enumerate}
\item $g(t)=\exp(tE)$ for $|t|\geq s$,
\item $g(\zeta)=\delta_{\Delta s}(u)$, where $\zeta:= \langle \delta_{\Delta s}(u),E(0)\rangle$,
\item $\mathrm{Lip}(g)\leq 1+\eta \Delta$,
\item $|(p\circ g)'(t)-p(E)|\leq C_{E}\Delta$ for all but finitely many $t\in \bbR$.
\end{enumerate}
Moreover, the same parameters $C_{E}, N_{E}, \Delta_{E}$ should work for any direction sufficiently close to $E$.

Notice that for $|t|\geq s$ the curve is explicitly defined by (1) and satisfies (3) and (4). Hence our task is to extend $g(t)$ for $-s<t<\zeta$ and $\zeta<t<s$. Since the two cases are similar, we show how to handle $-s<t<\zeta$. Up to left translations and reparameterizations of the curve, it suffices to verify the following equivalent claim.

\begin{claim}\label{bigclaim}
There exist $C_{E}, N_{E}$ and $\Delta_{E}\colon (0,\infty) \to (0,\infty)$ such that the following holds.

Given any $0<s<1$, $\eta>0$, $u\in \bbE_{n}$ with $d(u)\leq 1$ and $0<\Delta<\Delta_{E}(\eta)$, there is a Lipschitz horizontal curve $\varphi\colon [0,s+\zeta]\to \bbE_{n}$, where $\zeta:= \langle \delta_{\Delta s}(u),E(0)\rangle$, formed by joining at most $N_{E}$ horizontal lines, such that $\varphi(0)=0$, $\varphi(s+\zeta)=\exp(sE)\delta_{\Delta s}(u)$, and
\begin{enumerate}
\item[(A)] $\mathrm{Lip}(\varphi)\leq 1+\eta \Delta$,
\item[(B)] $|(p\circ \varphi)'(t)-p(E)|\leq C_{E}\Delta$ for all but finitely many $t\in \bbR$.
\end{enumerate}
Moreover, the same parameters $C_{E}, N_{E}, \Delta_{E}(\eta)$ work for any direction sufficiently close to $E$.
\end{claim}

\begin{proof}[Proof of Claim]
Since $E\neq \pm X_{2}$, we can write $E=aX_{1}+bX_{2}$, where $a^2+b^2=1$ and $a\neq 0$. Without loss of generality, up to changing the direction of the curve, we can assume $a>0$. Let $u:=u_{1}X_{1}+u_{2}X_{2}+\cdots + u_{n}X_{n}$. Identifying $\bbE_{n}$ and $\mathcal{E}_{n}$, since $E\in V_{1}$ we have
\[(sE)\delta_{\Delta s}(u)=\delta_{s}(E\delta_{\Delta}(u)),\]
and a simple computation gives
\[E+\delta_{\Delta}(u)=(a+\Delta u_{1})X_{1} + (b+\Delta u_{2})X_{2} + \Delta^{2}u_{3}X_{3} + \cdots + \Delta^{n-1}u_{n}X_{n},\]
\[[E, \delta_{\Delta}(u)]=\Delta(bu_{1}-au_{2})X_{3} - a\Delta^{2}u_{3}X_{4} - \cdots - a\Delta^{n-2}u_{n-1}X_{n}.\]
By the BCH formula \eqref{BCH}, it is then clear that $E \delta_{\Delta}(u)$ has the form
\[E \delta_{\Delta}(u)=(a+\Delta u_{1})X_{1} + (b+\Delta u_{2})X_{2} + \sum_{i=3}^{n} \eta_{i}X_{i},\]
where $\eta_{i}$ satisfy $|\eta_{i}|\leq \tilde{Q}\Delta$ for a constant $\tilde{Q}$ depending only on $n$. Next we write
\[E \delta_{\Delta}(u) = \delta_{a+\Delta u_{1}} \left( X_{1} + \frac{b+\Delta u_{2}}{a+\Delta u_{1}}X_{2} + \sum_{i=3}^{n} \frac{\eta_{i}}{(a+\Delta u_{1})^{i-1}}X_{i} \right),\]
and we define
\[ A:=\frac{b+\Delta u_{2}}{a+\Delta u_{1}}, \qquad a_{2}:=0, \qquad a_{i}:=\frac{\eta_{i}}{(a+\Delta u_{1})^{i-1}} \, \mbox{ for }3\leq i \leq n.\]
Let $A_{0}:=b/a$ and let $\varepsilon=\varepsilon(A_{0},n)$, $K=K(A_{0},n)$ and $N=N(n)$ be defined according to Lemma \ref{Filiformcurve}. By making $\Delta$ sufficiently small, we can ensure that $A\in (A_{0}-\varepsilon, A_{0}+\varepsilon)$ and $|a_{j}|<\varepsilon$. If we consider a direction $\tilde{E}=\tilde{a}X_{1}+\tilde{b}X_{2}$ with $\tilde{a}$ and $\tilde{b}$ sufficiently close to $a$ and $b$ (bound depending on $a$, $b$, $\varepsilon$), then we can ensure that if
\[\tilde{A}:=\frac{\tilde{b}+\Delta u_{2}}{\tilde{a}+\Delta u_{1}}, \qquad \tilde{a}_{2}:=0, \qquad \tilde{a}_{i}:=\frac{\eta_{i}}{(a+\Delta u_{1})^{i-1}} \, \mbox{ for }3\leq i \leq n,\]
then $\tilde{A}\in (A_{0}-\varepsilon, A_{0}+\varepsilon)$ and $|\tilde{a}_{j}|<\varepsilon$. Hence, in what follows, everything which applies to the direction $E$ will also apply to every direction $\tilde{E}$ sufficiently close to $E$ with the same parameters.

Applying Lemma \ref{Filiformcurve} with $E_{0}=X_{1}+AX_{2}$ and $E_{0}'=AX_{1}-X_{2}$ gives smooth functions $\theta_{1}, \ldots, \theta_{N}$ satisfying $|\theta_{i}|\leq K\max_{j\geq 2}|a_{j}|$ such that
\begin{equation}\label{eqfromFiliformcurve}
\prod_{i=1}^{N} \frac{1}{N} (E_{0}+\theta_{i}E_{0}') = E_{0}+a_{2}E_{0}'+a_{3}X_{3}+\cdots+a_{n}X_{n}.
\end{equation}
By definition of $a_{j}$, it follows that $|\theta_{i}|\leq Q\Delta$ for some constant $Q$ depending on $E$, provided that $\Delta$ is small compared to $a$. Using the definitions of $E_{0}$, $E_{0}'$ and dilating both sides of \eqref{eqfromFiliformcurve} by $a+\Delta u_{1}$ gives
\begin{align*}
&\prod_{i=1}^{N} \frac{1}{N} \Big( (a+\Delta u_{1})X_{1} + (b+\Delta u_{2})X_{2} + \theta_{i}( (b+\Delta u_{2})X_{1} - (a+\Delta u_{1})X_{2}) \Big)\\
&\qquad \qquad = (a+\Delta u_{1})X_{1} + (b+\Delta u_{2})X_{2} + \sum_{i=3}^{n} \eta_{i}X_{i}\\
&\qquad \qquad = E \delta_{\Delta}(u).
\end{align*}
Then, dilating both sides by $s$ and recalling that
\[\zeta= \langle \delta_{\Delta s}(u),E(0)\rangle=\Delta s(au_{1}+bu_{2}),\]
we get
\begin{align*}
&\prod_{i=1}^{N} \frac{s+\zeta}{N}\frac{s}{s+\zeta} \Big( (a+\Delta u_{1})X_{1} + (b+\Delta u_{2})X_{2} + \theta_{i}( (b+\Delta u_{2})X_{1} - (a+\Delta u_{1})X_{2}) \Big)\\
&\qquad \qquad = (sE)\delta_{\Delta s}(u).
\end{align*}
Define the horizontal curve $\varphi\colon [0,s+\zeta] \to \bbE_{n}$ by $\varphi(0)=0$ and
\[\varphi'(t)=\frac{s}{s+\zeta} \Big( (a+\Delta u_{1})X_{1} + (b+\Delta u_{2})X_{2} + \theta_{i}( (b+\Delta u_{2})X_{1} - (a+\Delta u_{1})X_{2}) \Big)\]
whenever
\[t\in I_{i}:=\left( \frac{(i-1)(s+\zeta)}{N},\, \frac{i(s+\zeta)}{N}\right), \qquad 1\leq i \leq N.\]
Then $\varphi(0)=0, \varphi(s+\zeta)=(sE)\delta_{\Delta s}(u)$ and $\varphi$ is a Lipschitz horizontal curve formed from joining $N$ horizontal lines. It remains to check that conditions (A) and (B) hold.

To verify (A), notice that by Lemma \ref{lipschitzhorizontal} it suffices to bound $|(p\circ \varphi)'|$. Recall that $|\theta_{i}|\leq Q\Delta$ for $1\leq i \leq N$ and $(1+x)^{1/2}\leq 1+x/2$ for $x\geq -1$. For any $t \in I_{i}$ and any sufficiently small $\Delta$ one has
\begin{align*}
|(p\circ \varphi)'| &= \frac{s}{s+\zeta}\Big|\Big( (a+\Delta u_{1})+\theta_{i}(b+\Delta u_{2}),\, (b+\Delta u_{2})-\theta_{i}(a+\Delta u_{1})\Big)\Big|\\
& =\frac{1}{1+\Delta(au_{1}+bu_{2})} \Big(1+2\Delta(au_{1}+bu_{2})+\Delta^{2}(u_{1}^2+u_{2}^2)\Big)^{1/2}\Big(1+\theta_{i}^2\Big)^{1/2}\\
& \leq \frac{1}{1+\Delta(au_{1}+bu_{2})} \Big( 1 + \Delta(au_{1}+bu_{2}) + \Delta^{2}(u_{1}^2+u_{2}^2)/2 \Big)\Big(1+\theta_{i}^2/2 \Big)\\
& \leq 1+\eta\Delta,
\end{align*}
where in the last inequality we used that $|\theta_{i}|\leq Q\Delta$ and made $\Delta$ sufficiently small relative to $\eta$. This proves (A).

To verify (B), we estimate as follows. For any $t\in I_{i}$ and any sufficiently small $\Delta$ one has
\begin{align*}
|(p\circ \varphi)'(t)-p(E)|&\leq C\Delta + \left| \frac{s}{s+\zeta}-1 \right|\\
&= C\Delta + \left| \frac{\zeta}{s+\zeta} \right|\\
&= C\Delta + \left| \frac{\Delta(au_{1}+bu_{2})}{1+\Delta(au_{1}+bu_{2})} \right| \\
&\leq C\Delta.
\end{align*}
This shows that (B) holds. Since our conclusions also hold for any direction sufficiently close to $E$, the claim is proved.
\end{proof}

As described earlier, Claim \ref{bigclaim} suffices to prove the theorem.
\end{proof}
 
Since deformability implies differentiability of the CC distance by Proposition \ref{deformimpliesdiff}, combining Proposition \ref{X2nogood} and Theorem \ref{deformFiliform} shows that a horizontal direction in a model filiform group $\bbE_{n}$ for $n\geq 4$ is deformable if and only if it is not $ \pm X_{2}$.

\section{Distances between piecewise linear curves}\label{sectiondistanceestimate}

In this section we prove a simple estimate for the distance between piecewise linear curves with similar directions in a general Carnot group $\bbG$. This will be needed to prove `almost maximality' implies differentiability. We first recall the following useful lemma \cite[Lemma 3.7]{Mag13}.

\begin{lemma}\label{Magnani}
Let $\bbG$ be a Carnot group of step $s$. Given any $\nu>0$, there exists a constant $K_{\nu}>0$ with the following property. If $N\in \bbN$ and $A_{j}, B_{j} \in \bbG$ defined for $j=1,\ldots,N$ satisfy $d(B_{j}B_{j+1}\cdots B_{N})\leq \nu$ and $d(A_{j},B_{j})\leq \nu$ for $j=1,\ldots,N$, then it holds that
\[d(A_{1}A_{2}\cdots A_{N}, B_{1}B_{2}\cdots B_{N}) \leq K_{\nu} \sum_{j=1}^{N} d(A_{j}, B_{j})^{1/s}.\]
\end{lemma}

Our estimate for the distance between curves is given by the following lemma.

\begin{lemma}\label{closedirectioncloseposition}
Let $\bbG$ be a Carnot group of step $s$. Then there is a constant $C_{a}\geq 1$ depending on $\bbG$ for which the following is true.

Suppose $E\in V_{1}$ with $\omega(E)\leq 1$, $0\leq D\leq 1$ and $N\in \bbN$. Let $g\colon (-R,R) \to \bbG$ be a Lipschitz horizontal curve with $g(0)=0$ satisfying the following conditions:
\begin{enumerate}
\item $g$ is formed by joining of at most $N$ horizontal lines,
\item $|(p\circ g)'(t)-p(E)|\leq D$ whenever $(p\circ g)'(t)$ exists.
\end{enumerate}
Then $d(g(t),\exp(tE)) \leq C_{a}ND^{1/s^2}|t|$ for every $t\in (-R,R)$.
\end{lemma}

\begin{proof}
Fix $t\geq 0$ without loss of generality. We may write
\[g(t)=\exp(t_{1}E_{1})\exp(t_{2}E_{2})\cdots \exp(t_{N}E_{N}),\]
where $t=t_{1}+t_{2}+\ldots +t_{N}$ with $t_{i}\geq 0$ and $E_{j}\in V_{1}$ with $|p(E_{j})-p(E)|\leq D$ for $1\leq j\leq N$. We intend to apply Lemma \ref{Magnani} to estimate
\[\frac{d(g(t),\exp(tE))}{|t|} = d\left(\exp\Big(\frac{t_{1}E_{1}}{t}\Big)\cdots \exp\Big(\frac{t_{N}E_{N}}{t}\Big),\, \exp\Big(\frac{t_{1}E}{t}\Big)\cdots \exp\Big(\frac{t_{N}E}{t}\Big)\right).\]
We first check that the hypotheses of Lemma \ref{Magnani} hold with the choice
\[A_{j}:=\exp ((t_{j}/t)E_{j}), \qquad B_{j}:=\exp((t_{j}/t)E), \qquad \nu:=3.\]
We first notice
\begin{align*}
d(B_{j}B_{j+1}\cdots B_{N}) &= d\left( \exp\left( \frac{(t_{j}+\ldots + t_{N})E}{t}\right) \right)\\
&= \frac{(t_{j}+\ldots + t_{N})}{t}d(\exp(E))\\
&\leq 1.
\end{align*}
Second, we can estimate as follows
\begin{align*}
d(A_{j},B_{j})&=(t_{j}/t)d(\exp(E_{j}),\exp(E))\\
&\leq d(\exp(E_{j}))+d(\exp(E))\\
&\leq 3.
\end{align*}
We can then combine Lemma \ref{Magnani} with Proposition \ref{euclideanheisenberg} to get
\begin{align*}
\frac{d(g(t),\exp(tE))}{|t|} &\leq K_{3} \sum_{j=1}^{N} d(\exp((t_{j}/t)E_{j}),\, \exp( (t_{j}/t)E))^{1/s}\\
& \leq K_{3} \sum_{j=1}^{N} d(\exp(E_{j}),\, \exp(E))^{1/s}\\
& \leq C \sum_{j=1}^{N} |\exp(E_{j})-\exp(E)|^{1/s^2}\\
& \leq C \sum_{j=1}^{N} |p(E_{j})-p(E)|^{1/s^2}\\
& \leq C ND^{1/s^2}.
\end{align*}
The proof is complete noticing that $C\geq 1$ is a constant depending only on $\bbG$.
\end{proof}

\section{Almost maximal directional derivatives and the UDS}\label{sectionUDS}

In this section we fix a Carnot group $\bbG$ satisfying the following condition.

\begin{assumptions}\label{ass}
Assume $\bbG \neq \bbR$. We say that $\bbG$ admits a ball of uniformly deformable directions if there exists an open ball $B\subset V_{1}$ of directions such that every $E\in B$ is deformable with the same parameters $C_{B}, N_{B}$ and $\Delta_{B}$.
\end{assumptions}

We will show that every Carnot group $\bbG$ which admits a ball of uniformly deformable directions contains a CC-Hausdorff dimension one (hence measure zero) set $N$ so that almost maximality of a directional derivative $Ef(x)$ implies differentiability if $x\in N$ and $E\in B$. Combining this with Proposition \ref{DoreMaleva} will lead to a proof of Theorem \ref{maintheorem}. By Theorem \ref{deformFiliform}, all model filiform groups $\bbE_{n}$  with $n\geq 2$ admit a ball of uniformly deformable directions. In particular, Theorem \ref{maintheorem} applies to Carnot groups of arbitrarily high step.

The Carnot group $\bbG$ will be identified with $\bbR^{n}$ by means of exponential coordinates of the first kind. Let $B_{\bbQ}$ denote the set of $E\in B$ with $\omega(E)=1$ which are a rational linear combination of the basis vectors $X_{1}, \ldots, X_{r}$ of $V_{1}$. Note that $B_{\bbQ}$ is dense in $B$ since the Euclidean sphere contains a dense set of points with rational coordinates. The construction of our universal differentiability set is given by the following lemma.

\begin{lemma}\label{uds}
For each choice of $E \in B_{\bbQ}$, $u\in \bbQ^{n}$ with $d(u)\leq 1$ and rationals
\[0<s<1, \qquad \eta>0, \qquad 0 < \Delta < \Delta_{B}(\eta),\]
let $\gamma_{E, u, s, \Delta, \eta}$ denote a curve granted by Definition \ref{deform} with parameters $C_{B}, N_{B}, \Delta_{B}(\eta)$. 

Let $L$ be the countable union of images of all translated curves $x\gamma_{E, u, s, \Delta}$, where $x\in \bbQ^{n}$ and $E, u, s, \Delta$ are as above. Then there is a $G_{\delta}$ set $N\subset \bbG$ containing $L$ which has Hausdorff dimension one with respect to the CC metric.
\end{lemma}

The proof of Lemma \ref{uds} is essentially the same as that of \cite[Lemma 5.4]{LPS17}. We also recall the following mean value type lemma for future use \cite[Lemma 3.4]{Pre90}.

\begin{lemma}\label{preissmeanvalue}
Suppose $|\zeta|<s<\rho$, $0<v<1/32$, $\sigma>0$ and $L>0$ are real numbers and let $\varphi, \psi\colon \mathbb{R} \to \mathbb{R}$ be Lipschitz maps satisfying $\mathrm{Lip}_{\mathbb{E}}(\varphi)+\mathrm{Lip}_{\mathbb{E}}(\psi)\leq L$, $\varphi(t)=\psi(t)$ for $|t|\geq s$ and $\varphi(\zeta)\neq \psi(\zeta)$. Suppose, moreover, that $\psi'(0)$ exists and that
\[|\psi(t)-\psi(0)-t\psi'(0)|\leq \sigma L|t|\]
whenever $|t|\leq \rho$,
\[\rho\geq s\sqrt{(sL)/(v|\varphi(\zeta)-\psi(\zeta)|)},\]
and
\[\sigma \leq v^{3}\Big( \frac{\varphi(\zeta)-\psi(\zeta)}{sL} \Big)^{2}.\]
Then there is $\tau\in (-s,s)\setminus \{\zeta\}$ such that $\varphi'(\tau)$ exists,
\[\varphi'(\tau)\geq \psi'(0)+v|\varphi(\zeta)-\psi(\zeta)|/s,\]
and
\[|(\varphi(\tau+t)-\varphi(\tau))-(\psi(t)-\psi(0))|\leq 4(1+20v)\sqrt{(\varphi'(\tau)-\psi'(0))L}|t|\]
for every $t\in \mathbb{R}$.
\end{lemma}

\begin{remark}\label{meanvalueremark}
By examining the proof of Lemma \ref{preissmeanvalue} in \cite{Pre90}. one can see that $\tau$ can additionally be chosen outside a given Lebesgue measure zero subset of $\mathbb{R}$.
\end{remark}

From now on we fix a set $N\subset \bbG$ as given by Lemma \ref{uds}. 

\begin{notation}\label{D^f}
For any Lipschitz function $f:\bbG \to \bbR$, define:
\[D^{f}:=\{ (x,E) \in N\times V_{1} \colon \omega(E)=1,\, Ef(x) \mbox{ exists}\}.\]
\end{notation}

\begin{theorem}\label{almostmaximalityimpliesdifferentiability}
Let $\bbG$ be a Carnot group of step $s$ which admit a ball of uniformly deformable directions (Assumptions \ref{ass}).

Let $f\colon \bbG \to \bbR$ be Lipschitz with $\mathrm{Lip}_{\bbG}(f) \leq 1/2$ and suppose $(x_{\ast}, E_{\ast})\in D^{f}$ with $E_{\ast}\in B$. Let $M$ denote the set of pairs $(x,E)\in D^{f}$ such that $Ef(x)\geq E_{\ast}f(x_{\ast})$ and for every $t\in (-1,1)$:
\begin{align*}
& |(f(x\exp(tE_{\ast}))-f(x)) - (f(x_{\ast}\exp(tE_{\ast}))-f(x_{\ast}))| \\
& \qquad \leq 6|t| (  (Ef(x)-E_{\ast}f(x_{\ast}))\mathrm{Lip}_{\bbG}(f))^{1/2s^{2}}.
\end{align*}
If
\[\lim_{\delta \downarrow 0} \sup \{Ef(x)\colon (x,E)\in M \mbox{ and }d(x,x_{\ast})\leq \delta\}\leq E_{\ast}f(x_{\ast}),\]
then $f$ is differentiable at $x_{\ast}$ and its Pansu differential is given by
\[L(x):=E_{\ast}f(x_{\ast})\langle x , E_{\ast}(0) \rangle=E_{\ast}f(x_{\ast})\langle p(x) , p(E_{\ast}) \rangle.\]
\end{theorem}

\begin{proof}

We assume $\mathrm{Lip}_{\bbG}(f)>0$, since otherwise the statement is trivial. Fix the following parameters:
\begin{enumerate}
\item $\varepsilon>0$ rational,
\item $0< v<1/32$ rational such that $4(1+20v)\sqrt{(2+v)/(1-v)}+v < 6$,
\item $\eta=\varepsilon v^{3}/3200$,
\item $\Delta_{B}(\eta/2)$, $C_{B}$ and $C_{a}$ according to Lemma \ref{closedirectioncloseposition} and Assumptions \ref{ass},
\item rational $0< \Delta < \min \{\eta v^2,\, \Delta_{B}(\eta/2),\, \Upsilon \}$, where
\[\Upsilon := \frac{\varepsilon v^{2s^{2}+1}}{8C_{B}^{2}C_{a}^{2s^{2}}N_{B}^{2s^2}\mathrm{Lip}_{\mathbb{G}}(f)^{2s^{2}-1}},\]
\item $\sigma=9\varepsilon^{2}v^{5}\Delta^2/256$,
\item $0<\rho<1$ such that
\begin{equation}\label{directionaldifferentiability}
|f(x_{\ast}\exp(tE_{\ast})) - f(x_{\ast})-tE_{\ast}f(x_{\ast})|\leq \sigma \mathrm{Lip}_{\mathbb{G}}(f)|t| \quad \mbox{for }|t|\leq \rho,
\end{equation}
\item $0<\delta < \rho \sqrt{3\varepsilon v\Delta^{3}}/4$ such that
\[Ef(x)<E_{\ast}f(x_{\ast})+\varepsilon v\Delta/2\]
whenever $(x,E)\in M$ and $d(x,x_{\ast})\leq 4\delta(1+1/\Delta)$.
\end{enumerate}
To prove Pansu differentiability of $f$ at $x_{\ast}$, we will show that
\[|f(x_{\ast}\delta_{t}(h))-f(x_{\ast})-tE_{\ast}f(x_{\ast})\langle u, E_{\ast}(0) \rangle |\leq \varepsilon t \qquad \mbox{for }d(u)\leq 1,\, 0<t<\delta.\]
Suppose this is not true. Then there exist $u\in \mathbb{Q}^{n}$ with $d(u) \leq 1$ and rational $0<r<\delta$ such that
\begin{equation}\label{badpoint}
|f(x_{\ast}\delta_{r}(u))-f(x_{\ast})-rE_{\ast}f(x_{\ast})\langle u, E_{\ast}(0) \rangle|> \varepsilon r.
\end{equation}
Let ${\rm{s}}=r/ \Delta \in \mathbb{Q}$. To contradict \eqref{badpoint}, we first construct Lipschitz horizontal curves $g$ and $h$ for which we can apply Lemma \ref{preissmeanvalue} with $\varphi:=f\circ g$ and $\psi:=f\circ h$.

\smallskip

\emph{Construction of $g$.} 

\smallskip

To ensure that the image of $g$ is a subset of the set $N$, we first introduce rational approximations to $x_{\ast}$ and $E_{\ast}$. Let
\begin{equation}\label{A1}
A_{1}:=\left( \frac{\eta \Delta}{C_{a}(N_{B}+2)}\right)^{s^2}
\end{equation}
and
\begin{equation}\label{A2}
A_{2}:=\Big(6- \Big( 4(1+20v) \Big( \frac{2+v}{1-v} \Big)^{1/2}+v\Big)\Big)^{s^{2}} \frac{(  \varepsilon v\Delta \mathrm{Lip}_{\mathbb{G}}(f)/2)^{1/2}}{C_{a}^{s^{2}}(N_{B}+2)^{s^2}\mathrm{Lip}_{\mathbb{G}}(f)^{s^{2}}}.
\end{equation}
Notice that $A_{1}, A_{2}>0$. We choose $\tilde{x}_{\ast}\in \mathbb{Q}^{n}$ and $\tilde{E}_{\ast}\in B_{\bbQ}$ sufficiently close to $x_{\ast}$ and $E_{\ast}$ respectively to ensure:
\begin{equation}\label{nowlistingthese} d(\tilde{x}_{\ast}\delta_{r}(u),x_{\ast})\leq 2r,\end{equation}
\begin{equation}\label{lista} d(\tilde{x}_{\ast}\delta_{r}(u), x_{\ast}\delta_{r}(u))\leq \sigma r,\end{equation}
\begin{equation}\label{listvector1} \omega(\tilde{E}_{\ast}-E_{\ast})\leq \min \{ \sigma, \, C_{B}\Delta,\, A_{1},\, A_{2} \}.\end{equation}
Recall that $0<r<\Delta$ and ${{\rm{s}}}=r/\Delta$ are rational and that $0<{\rm{s}}<1$. To construct $g$ we first apply Definition \ref{deform} with $E=\tilde{E}_{\ast}$ and parameters $\eta, {\rm{s}}, \Delta$, $\delta_{r}(u)$ and $u$ as defined above in \eqref{badpoint}. We then left translate this curve by $\tilde{x}_{\ast}$. This gives a Lipschitz horizontal curve $g\colon \bbR \to \bbG$ which is formed by joining at most $N_{B}$ horizontal lines such that
\begin{itemize}
\item $g(t)=\tilde{x}_{\ast}\exp(t\tilde{E}_{\ast})$ for $|t|\geq {\rm{s}}$,
\item $g(\zeta)=\tilde{x}_{\ast}\delta_{r}(u)$, where $\zeta := r\langle u,\tilde{E}_{\ast}(0)\rangle$,
\item $\mathrm{Lip}_{\mathbb{G}}(g)\leq 1+\eta\Delta$,
\item for all but finitely many $t\in \mathbb{R}$, $g'(t)$ exists and $|(p\circ g)'(t) - p(\tilde{E}_{\ast})| \leq C_{B}\Delta$ for $t\in \mathbb{R}$.
\end{itemize}
Since all the relevant quantities are chosen to be rational and $N$ is built according to Lemma \ref{uds}, it follows that the image of $g$ is contained in $N$.

\smallskip

\emph{Construction of $h$.} 

\begin{claim}
There exists a Lipschitz horizontal curve $h\colon \mathbb{R}\to \bbG$ such that
\[h(t)= \begin{cases} \tilde{x}_{\ast}\exp(t\tilde{E}_{\ast}) &\mbox{if }|t|\geq {\rm{s}},\\
x_{\ast}\exp(tE_{\ast}) &\mbox{if }|t|\leq {\rm{s}}/2, and
\end{cases}\]
in each of the regions $({\rm{s}}/2, {\rm{s}})$ and $(-{\rm{s}},-{\rm{s}}/2)$, $h$ is formed by joining at most $N_{B}$ horizontal lines. Moreover
\begin{itemize}
\item $\mathrm{Lip}_{\mathbb{G}}(h)\leq 1+\eta\Delta/2,$
\item for all but finitely many $t\in \mathbb{R}$, $h'(t)$ and satisfies the bound $|(p\circ h)'(t)-p(E_{\ast})|\leq \min \{A_{1}, A_{2}\}$.
\end{itemize}
\end{claim}

\begin{proof}[Proof of Claim]
Up to a left translation, we may start by assuming that $x_{\ast}=0$. Clearly $h(t)$ is defined explicitly and satisfies the required conditions for $|t|\leq {\rm{s}}/2$ and $|t|\geq {\rm{s}}$. We now show how to extend $h$ in $({\rm{s}}/2,{\rm{s}})$. The extension to $(-{\rm{s}},-{\rm{s}}/2)$ is essentially the same.

Recall $\Delta_{B}(1)$ and $C_{B}$ from Assumptions \ref{ass}. Choose $0<\Gamma<\Delta(1)$ satisfying
\[(1+\Gamma)^{2}\leq 1+\eta\Delta/2\]
and
\[C_{B}\Gamma(1+\Gamma)+\Gamma \leq \min \{A_{1}, A_{2}\}.\]
Define $\lambda={\rm{s}}\Gamma/2<\Gamma$. Choose $v\in \bbG$ with $d(v)\leq 1$ such that
\[\delta_{\lambda}(v)=\exp(-{\rm{s}}E_{\ast})\tilde{x}_{\ast}\exp({\rm{s}}\tilde{E}_{\ast}).\]
This is possible if the rational approximation introduced earlier is chosen correctly; note that the rational approximation was introduced after all quantities upon which $\lambda$ depends. We now apply Remark \ref{deform2} with
\begin{itemize}
\item $\eta=1$ and $\Delta$ replaced by $\Gamma$,
\item ${\rm{s}}$ replaced by ${\rm{s}}/2$,
\item $u$ replaced by $v$,
\item $\zeta$ replaced by $\tilde{\zeta}: = \langle \delta_{{\rm{s}}\Gamma /2}(v),E_{\ast}(0)\rangle$.
\end{itemize}
We obtain a Lipschitz horizontal curve $\varphi: [0,{\rm{s}}/2+\tilde{\zeta}]\to \bbG$ that is formed by joining at most $N_{B}$ horizontal lines such that
\begin{itemize}
\item $\varphi(0)=0$,
\item $\varphi({\rm{s}}/2+\tilde{\zeta})=\exp(-({\rm{s}}/2)E_{\ast})\tilde{x}_{\ast}\exp({\rm{s}}\tilde{E}_{\ast})$,
\item $\mathrm{Lip}_{\bbG}(\varphi)\leq 1+\Gamma$,
\item $\varphi'(t)$ exists and $|(p\circ \varphi)'(t)-p(E_{\ast})|\leq C_{B}\Gamma$ for all except finitely many $t\in [0,{\rm{s}}/2+\tilde{\zeta}]$.
\end{itemize}
Then $\tilde{\varphi}:[0,1]\to \bbG$ defined by $\tilde{\varphi}(t)=\varphi(({\rm{s}}/2+\tilde{\zeta})t)$ is a Lipschitz horizontal curve such that
\begin{itemize}
\item $\tilde{\varphi}(0)=0$,
\item $\tilde{\varphi}(1)=\exp(-({\rm{s}}/2)E_{\ast})\tilde{x}_{\ast}\exp({\rm{s}}\tilde{E}_{\ast})$,
\item $\mathrm{Lip}_{\bbG}(\tilde{\varphi})\leq (1+\Gamma)({\rm{s}}/2+\tilde{\zeta})$,
\item $\tilde{\varphi}'(t)$ exists and $|(p\circ \tilde{\varphi})'(t)-({\rm{s}}/2+\tilde{\zeta})p(E_{\ast})|\leq C_{B}\Gamma({\rm{s}}/2+\tilde{\zeta})$ for all but finitely many $t\in [0,1]$.
\end{itemize}
Define $h_1 :[{\rm{s}}/2,{\rm{s}}]\to \bbG$ by
\[h_1(t)=\exp(({\rm{s}}/2)E_{\ast}) \tilde{\varphi}((2/{\rm{s}})(t-{\rm{s}}/2)).\]
Then $h_1$ is a Lipschitz horizontal curve which satisfies $h_{1}({\rm{s}}/2)=\exp(({\rm{s}}/2)E_{\ast})$ and $h_{1}({\rm{s}})=\tilde{x}_{\ast}\exp({\rm{s}}\tilde{E}_{\ast})$. Note that $|p(v)|\leq d(v)\leq 1$ implies $|\tilde{\zeta}|\leq \lambda$. Hence we have
\begin{align*}
\mathrm{Lip}_{\bbG}(h_1)&\leq \frac{2(1+\Gamma)({\rm{s}}/2+\tilde{\zeta})}{{\rm{s}}}\\
& \leq \frac{2(1+\Gamma)({\rm{s}}/2+\lambda)}{{\rm{s}}}\\
&\leq (1+\Gamma)^{2}\\
&\leq 1+\eta\Delta/2.
\end{align*}
Then, for all but finitely many $t\in [{\rm{s}}/2,{\rm{s}}]$
\[|(p\circ h_1)'(t)-(1+2\tilde{\zeta}/{\rm{s}})p(E_{\ast})| \leq C_{B}\Gamma (1+2\tilde{\zeta}/{\rm{s}}),\]
and this implies
\begin{align*}
|(p\circ h_1)'(t)-p(E_{\ast})| &\leq C_{B}\Gamma(1+2\tilde{\zeta}/{\rm{s}}) + 2|\tilde{\zeta}|/{\rm{s}}\\
&\leq C_{B}\Gamma(1+\Gamma)+\Gamma\\
&\leq \min \{A_{1}, A_{2}\}.
\end{align*}

Defining $h(t):=h_{1}(t)$ for any $t\in [{\rm{s}}/2,{\rm{s}}]$ we obtain the desired properties. The extension of $h$ in $[-{\rm{s}},-{\rm{s}}/2]$ is similar.
\end{proof}

\smallskip

\emph{Application of Lemma \ref{preissmeanvalue}.} 

\smallskip

We now prove that the hypotheses of Lemma \ref{preissmeanvalue} hold with $L:=(2+\eta \Delta)\mathrm{Lip}_{\bbG}(f)$, $\varphi:=f\circ g$ and $\psi:=f\circ h$. The inequalities $|\zeta|<{\rm{s}}<\rho$, $0<v<1/32$ and the equality $\varphi(t)=\psi(t)$ for $|t|\geq {\rm{s}}$ are clear. Since $\mathrm{Lip}_{\bbG}(g), \mathrm{Lip}_{\bbG}(h) \leq 1+\eta\Delta/2$, we have $\mathrm{Lip}_{\mathbb{E}}(\varphi)+\mathrm{Lip}_{\mathbb{E}}(\psi)\leq L$. 

Notice that \eqref{lista} implies
\[ |f(\tilde{x}_{\ast}\delta_{r}(u)) - f(x_{\ast}\delta_{r}(u))| \leq \sigma r \mathrm{Lip}_{\bbG}(f).\]
Since $|\zeta|\leq r\leq \rho$, we may evaluate \eqref{directionaldifferentiability} at $t=\zeta$ to obtain
\begin{align*}
|f(x_{\ast}\exp(\zeta E_{\ast}))-f(x_{\ast})-\zeta E_{\ast}f(x_{\ast})| &\leq \sigma \mathrm{Lip}_{\bbG}(f)|\zeta| \\
&\leq \sigma r\mathrm{Lip}_{\bbG}(f).
\end{align*}
Next, note that \eqref{listvector1} implies $|\tilde{E}_{\ast}(0)-E_{\ast}(0)|\leq \sigma$. Recalling that $\zeta=r\langle u, \tilde{E}_{\ast}(0)\rangle$ we can estimate as follows:
\begin{align*}
|\zeta E_{\ast}f(x_{\ast}) - r\langle u,E_{\ast}(0)\rangle E_{\ast}f(x_{\ast})| & = r|E_{\ast}f(x_{\ast})||\langle u, \tilde{E}_{\ast}(0) -E_{\ast}(0)\rangle|\\
&\leq r\mathrm{Lip}_{\bbG}(f)|\tilde{E}_{\ast}(0)-E_{\ast}(0)|\\
&\leq \sigma r \mathrm{Lip}_{\bbG}(f).
\end{align*}
Hence we obtain,
\begin{equation}\label{yetanother}|f(x_{\ast}\exp(\zeta E_{\ast}))-f(x_{\ast})-r\langle u,E_{\ast}(0)\rangle E_{\ast}f(x_{\ast})| \leq 2\sigma r\mathrm{Lip}_{\bbG}(f).\end{equation}

Since $|\zeta|\leq r=\Delta {\rm{s}}\leq {\rm{s}}/2$ we have $h(\zeta)=x_{\ast}\exp(\zeta E_{\ast})$. The definition of $g$ gives $g(\zeta)=\tilde{x}_{\ast}\delta_{r}(u)$. Using also \eqref{badpoint} and \eqref{yetanother}, we can estimate as follows:
\begin{align}\label{differenceofcomposition}
|\varphi(\zeta)-\psi(\zeta)|&= |f(g(\zeta)) - f(h(\zeta))| \nonumber \\
&= |f(\tilde{x}_{\ast}\delta_{r}(u)) - f(x_{\ast}\exp(\zeta E_{\ast}))| \nonumber\\
& \geq |f(x_{\ast}\delta_{r}(u)) - f(x_{\ast}\exp(\zeta E_{\ast}))| - |f(\tilde{x}_{\ast}\delta_{r}(u)) - f(x_{\ast}\delta_{r}(u))|\nonumber \\
& \geq |f(x_{\ast}\delta_{r}(u))-f(x_{\ast})-rE_{\ast}f(x_{\ast})\langle u, E_{\ast}(0) \rangle | \nonumber \\
& \quad -|f(x_{\ast}\exp(\zeta E_{\ast})) - f(x_{\ast}) - r E_{\ast}f(x_{\ast})\langle u, E_{\ast}(0) \rangle| \nonumber \\
& \quad - \sigma r\mathrm{Lip}_{\bbG}(f) \nonumber\\
&\geq \varepsilon r - 2\sigma r\mathrm{Lip}_{\bbG}(f) -  \sigma r\mathrm{Lip}_{\bbG}(f)\nonumber \\
&= \varepsilon r - 3\sigma r\mathrm{Lip}_{\bbG}(f) \nonumber \\
&\geq 3\varepsilon r/4.
\end{align}
In particular, $\varphi(\zeta)\neq \psi(\zeta)$.

The derivative $\psi'(0)$ exists and equals $E_{\ast}f(x_{\ast})$, since $\psi(t)=f(x_{\ast}\exp(tE_{\ast}))$ for every $|t|\leq {\rm{s}}/2$. We next check that
\begin{equation}\label{psiprime}
|\psi(t)-\psi(0)-t\psi'(0)| \leq \sigma L|t| \quad \mbox{for }|t|\leq \rho.
\end{equation}
Recall that $h(0)=x_{\ast}$, $|(p\circ h)'-p(E_{\ast})|\leq A_{1}$ (see \eqref{A1} for the definition of $A_{1}$) and $h$ is formed by joining at most $N_{B}+2$ horizontal lines. Hence Lemma \ref{closedirectioncloseposition} implies that
\[ d(x_{\ast}\exp(tE_{\ast}),h(t))\leq C_{a}(N_{B}+2)A_{1}^{1/s^2}|t|\leq \eta\Delta |t| \quad \mbox{ for every }t\in \bbR.\]
Hence, using also \eqref{directionaldifferentiability} and $L=(2+\eta\Delta)\mathrm{Lip}_{\bbG}(f)$, one has
\begin{align*}
|\psi(t)-\psi(0)-t\psi'(0)| &\leq |f(x_{\ast}\exp(tE_{\ast})) - f(x_{\ast})-tE_{\ast}f(x_{\ast})|\\
& \qquad + |f(x_{\ast}\exp(tE_{\ast})) - f(h(t))|\\
&\leq \sigma \mathrm{Lip}_{\bbG}(f)|t| + \mathrm{Lip}_{\bbG}(f)\eta\Delta|t|\\
&\leq \sigma L |t| \quad \mbox{ for }|t|\leq \rho.
\end{align*}

Since $\mathrm{Lip}_{\bbG}(f)\leq 1/2$ we have $L\leq 4$. By using $r< \delta$, ${\rm{s}}=r/\Delta$, \eqref{differenceofcomposition} and the definition of $r, \delta, \Delta$ and ${\rm{s}}$ we deduce
\begin{align*}
{\rm{s}}\sqrt{ {\rm{s}}L/(v|\varphi(\zeta)-\psi(\zeta)|)}  &\leq 4{\rm{s}}\sqrt{{\rm{s}}/(3\varepsilon rv)}\\
&= 4r/\sqrt{3\varepsilon v\Delta^3}\\
&\leq 4\delta/ \sqrt{3\varepsilon v\Delta^{3}}\\
&\leq \rho.
\end{align*}

Finally we use \eqref{differenceofcomposition}, $L\leq 4$ and the definition of $\sigma$ to get
\begin{align*}
v^3 (|\varphi(\zeta)-\psi(\zeta)|/({\rm{s}}L))^2&\geq v^3(3\varepsilon r / 16{\rm{s}})^2\\
&= 9\varepsilon^2 v^3 \Delta^2 /256\\
& \geq \sigma.
\end{align*}

We can now apply Lemma \ref{preissmeanvalue}. We obtain $\tau \in (-{\rm{s}},{\rm{s}})\setminus \{\zeta \}$ such that $\varphi'(\tau)$ exists and satisfies
\begin{equation}\label{bigderivative}
\varphi'(\tau)\geq \psi'(0)+v|\varphi(\zeta)-\psi(\zeta)|/{\rm{s}},
\end{equation}
and for every $t\in \mathbb{R}$:
\begin{equation}\label{incrementsbound}
|(\varphi(\tau+t)-\varphi(\tau))-(\psi(t)-\psi(0))|\leq 4(1+20v)\sqrt{(\varphi'(\tau)-\psi'(0))L}|t|
\end{equation}
Since $g$ is a horizontal curve, we may use Remark \ref{meanvalueremark} to additionally choose $\tau$ such that $g'(\tau)$ exists and is in $\mathrm{Span}\{X_{i}(g(\tau))\colon 1\leq i\leq r\}$. 

\smallskip

\emph{Conclusion.}

\smallskip

Let $x:=g(\tau)\in N$ and choose $E\in V_{1}$ with $E(g(\tau))=g'(\tau)/|p(g'(\tau))|$, which implies that $\omega(E)=1$. From \eqref{bigderivative} and \eqref{incrementsbound} we will obtain
\begin{equation}\label{betterpair1}
Ef(x)\geq E_{\ast}f(x_{\ast}) + \varepsilon v\Delta/2,
\end{equation}
\begin{equation}\label{betterpair2}
(x,E)\in M.
\end{equation}
We first observe that this suffices to conclude the proof. Indeed, by \eqref{nowlistingthese} and since $g(\zeta)=\tilde{x}_{\ast}\delta_{r}(u)$ one has
\begin{align*}
d(x,x_{\ast}) &\leq d(g(\tau),g(\zeta))+d(\tilde{x}_{\ast}\delta_{r}(u),x_{\ast})\\
&\leq \mathrm{Lip}_{\bbG}(g)|\tau - \zeta| +2r\\
&\leq 4({\rm{s}}+r)\\
&= 4r(1+1/\Delta)\\
&\leq 4\delta(1+1/\Delta).
\end{align*}
Since $x\in N$, combining this with \eqref{betterpair1} and \eqref{betterpair2} contradicts the choice of $\delta$. This forces us to conclude that \eqref{badpoint} is false, finishing the proof.

\smallskip

\emph{Proof of \eqref{betterpair1}.}

\smallskip

Using \eqref{differenceofcomposition} and \eqref{bigderivative} we have that
\begin{equation}\label{stanco}
\varphi'(\tau)-\psi'(0)\geq 3\varepsilon vr/4{\rm{s}}=3\varepsilon v\Delta/4.
\end{equation}
Notice that, by the definition of $E$, by Definition \ref{defdirectionalderivative}, and the fact that g is a concatenation of horizontal lines, we have $\varphi'(\tau)=Ef(x)|p(g'(\tau))|$. Since $\omega(E)=1$, we deduce that $|\varphi'(\tau)|/|p(g'(\tau))|\leq \mathrm{Lip}_{\bbG}(f)$. Similarly $|p(g'(\tau))| \leq \mathrm{Lip}_{\bbG}(g)\leq 1+\eta \Delta$. Since $\psi'(0)=E_{\ast}f(x_{\ast})$, by \eqref{stanco} we have
\begin{align*}
& Ef(x)-E_{\ast}f(x_{\ast})-(1-v)(\varphi'(\tau)-\psi'(0))\\
&\qquad = v(\varphi'(\tau)-\psi'(0)) + (1-|p(g'(\tau))|)\varphi'(\tau)/|p(g'(\tau))|\\
&\qquad \geq 3\varepsilon v^2\Delta/4 - \eta\Delta \mathrm{Lip}_{\bbG}(f)\\
&\qquad \geq 0.
\end{align*}
In the last inequality we used $\mathrm{Lip}_{\bbG}(f)\leq 1/2$ and $\eta\leq 3\varepsilon v^2 /2$. From this we use $0<v<1/32$ and \eqref{stanco} again to deduce
\begin{equation}\label{noimagination}
Ef(x)-E_{\ast}f(x_{\ast})\geq (1-v)(\varphi'(\tau)-\psi'(0))\geq \varepsilon v\Delta /2,
\end{equation}
which proves \eqref{betterpair1}.

\smallskip

\emph{Proof of \eqref{betterpair2}.}

\smallskip

Recall that $|(p\circ g)'(t)-p(\tilde{E}_{\ast})| \leq C_{B}\Delta$ for all but finitely many $t$. Using \eqref{listvector1}, this implies $|(p\circ g)'(t)- p(E_{\ast})|\leq 2C_{B}\Delta$ for all but finitely many $t$. Since $x=g(\tau)$ and $g$ is formed by joining at most $N_{B}$ horizontal lines, we can apply Lemma \ref{closedirectioncloseposition} to obtain
\[d(g(\tau+t),x\exp(tE_{\ast}))\leq C_{a}N_{B}(2C_{B}\Delta)^{1/s^{2}}|t| \qquad \mbox{for every }t\in \bbR.\]
By \eqref{noimagination} we have $\Delta \leq 2(Ef(x)-E_{\ast}f(x_{\ast}))/(\varepsilon v)$. Combining this fact with the definition of $\Delta$, we deduce that
\begin{align}\label{add1}
&|(f(x\exp(tE_{\ast}))-f(x))-(f(g(\tau+t))-f(g(\tau)))|\nonumber \\
&\qquad = |f(x\exp(tE_{\ast}))-f(g(\tau+t))|\nonumber \\
&\qquad \leq \mathrm{Lip}_{\bbG}(f)  d(g(\tau+t), x\exp(tE_{\ast}))\nonumber \\
&\qquad \leq C_{a}N_{B}(2C_{B}\Delta)^{1/s^{2}} \mathrm{Lip}_{\bbG}(f)|t| \nonumber \\
&\qquad \leq C_{a}N_{B}(2C_{B}\sqrt{\Delta})^{1/s^{2}}\mathrm{Lip}_{\bbG}(f)|t| \Big( \frac{2(Ef(x)-E_{\ast}f(x_{\ast}))}{\varepsilon v} \Big)^{\frac{1}{2s^{2}}}\nonumber \\
&\qquad \leq v|t|\big((Ef(x)-E_{\ast}f(x_{\ast}))\mathrm{Lip}_{\bbG}(f) \big)^{\frac{1}{2s^{2}}} \Big(\frac{8C_{B}^{2}C_{a}^{2s^{2}}N_{B}^{2s^{2}}\Delta\mathrm{Lip}_{\bbG}(f)^{2s^{2}-1}}{\varepsilon v^{2s^{2}+1}} \Big)^{\frac{1}{2s^{2}}} \nonumber \\
&\qquad \leq v|t|\big((Ef(x)-E_{\ast}f(x_{\ast}))\mathrm{Lip}_{\bbG}(f) \big)^{\frac{1}{2s^{2}}} \quad \mbox{ for }t\in \bbR.
\end{align}
Combining \eqref{incrementsbound}, \eqref{noimagination} and $L=(2+\eta \Delta)\mathrm{Lip}_{\bbG}(f)\leq (2+v)\mathrm{Lip}_{\bbG}(f)$ gives
\begin{align}\label{add2}
&|(\varphi(\tau+t)-\varphi(\tau))-(\psi(t)-\psi(0))| \nonumber \\
&\qquad \leq 4(1+20v)|t| \Big( \frac{(2+v)\mathrm{Lip}_{\bbG}(f)(Ef(x)-E_{\ast}f(x_{\ast}))}{1-v} \Big)^{\frac{1}{2}} \quad \mbox{ for }t\in \bbR.
\end{align}
Since $\mathrm{Lip}_{\bbG}(f)\leq 1/2$, we easily get
\[((Ef(x)-E_{\ast}f(x_{\ast}))\mathrm{Lip}_{\bbG}(f))^{\frac{1}{2}} \leq ((Ef(x)-E_{\ast}f(x_{\ast}))\mathrm{Lip}_{\bbG}(f))^{\frac{1}{2s^{2}}}\]
since both sides are less than $1$. Hence adding \eqref{add1} and \eqref{add2} and using the definition $\varphi=f\circ g$ gives for $t\in \bbR$:
\begin{align}\label{add3}
& |f(x\exp(tE_{\ast})-f(x))-(\psi(t)-\psi(0))|\nonumber \\
&\qquad \leq \Big( 4(1+20v) \Big( \frac{2+v}{1-v} \Big)^{\frac{1}{2}}+v\Big) |t|( (Ef(x)-E_{\ast}f(x_{\ast}))\mathrm{Lip}_{\bbG}(f))^{\frac{1}{2s^{2}}}.
\end{align}

Recall that $\psi = f\circ h$ and that $h$ is a concatenation of at most $N_B+2$ horizontal lines such that $h(0)=x_{\ast}$ and the inequality $|(p\circ h)'-p(E_{\ast})|\leq A_{2}$ holds for all but finitely many $t\in \bbR$. Then, by Lemma \ref{closedirectioncloseposition}, \eqref{A2} and \eqref{betterpair1}, we have
\begin{align}\label{add4}
&|(\psi(t)-\psi(0))-(f(x_{\ast}\exp(tE_{\ast}))-f(x_{\ast}))| \nonumber \\
&\quad =|f(h(t))-f(x_{\ast}\exp(tE_{\ast}))| \nonumber \\
&\quad \leq \mathrm{Lip}_{\bbG}(f)d(h(t), x_{\ast}\exp(tE_{\ast})) \nonumber \\
&\quad \leq \mathrm{Lip}_{\bbG}(f)C_{a}(N_{B}+2)A_{2}^{1/s^{2}}|t| \nonumber \\
&\quad =  \Big(6- \Big( 4(1+20v) \Big( \frac{2+v}{1-v} \Big)^{\frac{1}{2}}+v\Big)\Big)|t|( \varepsilon v \Delta/2)\mathrm{Lip}_{\bbG}(f))^{\frac{1}{2s^{2}}} \nonumber \\
&\quad \leq  \Big(6- \Big( 4(1+20v) \Big( \frac{2+v}{1-v} \Big)^{\frac{1}{2}}+v\Big)\Big)|t|((Ef(x)-E_{\ast}f(x_{\ast}))\mathrm{Lip}_{\bbG}(f))^{\frac{1}{2s^{2}}} \quad \mbox{ for }t\in \mathbb{R}.
\end{align}
Adding \eqref{add3} and \eqref{add4} gives for every $t\in \bbR$
\begin{align*}
& |(f(x\exp(tE_{\ast}))-f(x)) - (f(x_{\ast}\exp(tE_{\ast}))-f(x_{\ast}))| \\
& \qquad \leq 6|t| \big(  (Ef(x)-E_{\ast}f(x_{\ast}))\mathrm{Lip}_{\bbG}(f) \big)^{\frac{1}{2s^{2}}}.
\end{align*}
This implies \eqref{betterpair2}, hence proving the theorem.
\end{proof}

\section{Construction of an almost maximal directional derivative}\label{sectionconstruction}
Assume $\mathbb{G}$ is a Carnot group of step $s$, rank $r$ and topological dimension $n$. Fix a $G_{\delta}$ set $N\subset \mathbb{G}$. The main result of this section is Proposition \ref{DoreMaleva}, which is an adaptation of \cite[Theorem 3.1]{DM11} and of \cite[Theorem 6.1]{PS16} to $\mathbb{G}$. It shows that given a Lipschitz function $f_{0}\colon \mathbb{G} \to \mathbb{R}$, there is a Lipschitz function $f\colon \mathbb{G} \to \mathbb{R}$ such that $f-f_{0}$ is $\mathbb{G}$-linear and $f$ has an almost locally maximal horizontal directional derivative at a point of $N$.


\begin{notation}\label{D}
For any Lipschitz function $f:\bbG \to \bbR$, define
\[D^{f}:=\{ (x,E) \in N\times V_{1} \colon \omega(E)=1,\, Ef(x) \mbox{ exists}\}.\]
\end{notation}

Note that if $f-f_{0}$ is $\mathbb{G}$-linear then $D^{f}=D^{f_{0}}$ and also the functions $f$ and $f_{0}$ have the same points of Pansu differentiability.

\begin{proposition}\label{DoreMaleva}
Suppose $f_0:\mathbb{G}\to \mathbb{R}$ is a Lipschitz function, $(x_0,E_0)\in D^{f_0}$ and $\delta_0, \mu, \tau, K>0$. Then there is a Lipschitz function $f:\mathbb{G}\to \mathbb{R}$ such that $f-f_0$ is $\mathbb{G}$-linear with $\mathrm{Lip}_{\mathbb{G}}(f-f_{0})\leq \mu$, and a pair $(x_{\ast},E_{\ast})\in D^{f}$ with $d(x_{\ast},x_0)\leq \delta_0$ and $\omega(E_{\ast}-E_0)\leq \tau$ such that $E_{\ast}f(x_{\ast})>0$ is almost locally maximal in the following sense.

For any $\varepsilon>0$ there is $\delta_{\varepsilon}>0$ such that, whenever $(x,E)\in D^{f}$ satisfies both:
\begin{enumerate}
\item $d(x,x_{\ast})\leq \delta_{\varepsilon}$, $Ef(x)\geq E_{\ast}f(x_{\ast})$ and
\item for any $t\in (-1,1)$
\begin{align*}
&|(f(x\exp(tE_{\ast}))-f(x))-(f(x_{\ast}\exp(tE_{\ast}))-f(x_{\ast}))|\\
& \qquad \leq K|t| ( Ef(x)-E_{\ast}f(x_{\ast}))^{\frac{1}{2s^2}},
\end{align*}
\end{enumerate}
then
\[Ef(x)<E_{\ast}f(x_{\ast})+\varepsilon.\]
\end{proposition}
We use the remainder of this section to prove Proposition \ref{DoreMaleva}. 
We recall the following constants:
\begin{itemize}
\item $C_{\mathrm{a}}\geq 1$ chosen as in Lemma \ref{closedirectioncloseposition},
\item $C_D \geq 1$ as in \ref{conjugatedistance}.
\end{itemize}
Fix $f_{0}, x_0, E_0, \delta_0, \tau, \mu, K$ as given in the statement of Proposition \ref{DoreMaleva} and define $t_0:=\min \{1/4,\, \mu/2\}$.

\begin{assumptions}\label{Ass}
Without loss of generality, we make the following assumptions:
\begin{itemize}
\item $K\geq 4s^2$, since increasing $K$ makes the statement of Proposition \ref{DoreMaleva} stronger,
\item $\mathrm{Lip}_{\mathbb{G}}(f_0)\leq \min\{1/2, t_0 \tau^2 / 32\}$, after multiplying $f_0$ by a positive constant and possibly increasing $K$,
\item $E_0f(x_0)\geq 0$, by replacing $E_0$ by $-E_0$ if necessary.
\end{itemize}
\end{assumptions}

We prove Proposition \ref{DoreMaleva} using a technique similar to the one implemented in \cite[Theorem 6.1]{PS16}, namely by using Algorithm \ref{alg} below to construct a sequence of Lipschitz functions $(f_m)$ and a sequence of pairs $(x_m,E_m)$ in $D^{f_m}$ so that $E_{m}f(x_m)$ converges to an almost locally maximal directional derivative for $f$. More precisely, we show that the limits $(x_{\ast},E_{\ast})$ and $f$ have the properties stated in Proposition \ref{DoreMaleva}.

\begin{notation}\label{comparison}
Suppose $h:\mathbb{G}\to\mathbb{R}$ is Lipschitz, the pairs $(x,E)$ and $(x',E')$ belong to $D^h$, and $\sigma \geq 0$. We write
\[(x,E)\leq_{(h,\sigma)} (x',E')\]
if $E h(x)\leq E' h(x')$ and for all $t\in (-1,1)$
\begin{align*}
&|(h(x'\exp(tE))-h(x'))-(h(x\exp(tE))-h(x))|\\
&\qquad \leq K (\sigma+ (E'f(x')-Ef(x))^{\frac{1}{2s^2}})|t|.
\end{align*}
\end{notation}

In the language of Notation \ref{comparison}, Proposition \ref{DoreMaleva}(2) means $(x_{\ast},E_{\ast})\leq_{(f,0)} (x,E)$. 

Since $N$ is $G_{\delta}$ we can fix open sets $U_k\subset \mathbb{G}$ such that $N=\cap_{k=0}^{\infty} U_k$. We may assume that $U_{0}=\mathbb{G}$. We point out that, in Algorithm \ref{alg} below, the order in which the parameters are chosen plays a crucial role in what follows.

\begin{algorithm}\label{alg}
Let $f_0, x_0, E_0, \tau$ and $\delta_0$ be as in the assumptions of Proposition \ref{DoreMaleva}. Let $\sigma_0:=2$ and $t_0:=\min \{1/4,\, \mu/2\}$.

Then we can recursively define
\begin{enumerate}
\item $f_m(x):=f_{m-1}(x)+t_{m-1} \langle x, E_{m-1}(0) \rangle$,
\item $\sigma_m\in (0, \sigma_{m-1}/4)$,
\item $t_m\in (0, \min\{t_{m-1}/2,\, \sigma_{m-1}/(s^2 m)\})$,
\item $\lambda_m\in (0, \min\{t_m\sigma_m^{2s^2}/(2C_{\mathrm{a}}^{2s^2}),\, t_m\tau^2/2^{2m+3}\})$,
\item $D_m$ to be the set of pairs $(x,E)\in D^{f_m}=D^{f_0}$ such that $d(x,x_{m-1})<\delta_{m-1}$ and
\[(x_{m-1}, E_{m-1})\leq_{(f_m,\sigma_{m-1}-\varepsilon)} (x,E)\]
for some $\varepsilon\in (0,\sigma_{m-1})$,
\item $(x_m,E_m)\in D_m$ such that $Ef_m(x)\leq E_mf_m(x_m)+\lambda_m$ for every pair $(x,E)\in D_m$,
\item $\varepsilon_m\in (0,\sigma_{m-1})$ such that $(x_{m-1}, E_{m-1})\leq_{(f_m,\sigma_{m-1}-\varepsilon_m)} (x_m, E_m)$,
\item $\delta_m\in (0, (\delta_{m-1}-d(x_m,x_{m-1}))/2)$ such that $\overline{B_{\mathbb{G}}(x_m,\delta_m)}\subset U_m$ and for all $|t|<C_D\delta_m^{\frac{1}{s}}/\varepsilon_m$
\begin{align*}
&|(f_m(x_m\exp(tE_{m}))-f_m(x_m))-(f_m(x_{m-1}\exp(tE_{m-1}))-f_m(x_{m-1}))| \\
&\qquad \leq( E_mf_m(x_m)-E_{m-1}f_m(x_{m-1})+\sigma_{m-1}) |t|.
\end{align*}
\end{enumerate}
\end{algorithm}

\begin{proof}
Clearly one can make choices satisfying (1)--(5). For (6)--(8) we can proceed exactly as in \cite[Proof of Algorithm 6.4]{PS16} using Lemma \ref{lemmascalarlip} and Lemma \ref{lipismaximal} instead of \cite[Lemma 5.2]{PS16} and \cite[Lemma 3.3]{PS16}, respectively.



\end{proof}

We omit the proof of the following Lemma since it is exactly the same as the one of \cite[Lemma 6.5]{PS16} for the Heisenberg group. 

\begin{lemma} \label{inclusionballs}
The sequences $\sigma_m, t_m, \lambda_m, \delta_m, \varepsilon_m$ converge to $0$, and for every $m\geq 1$ the inclusion
\[\overline{B_{\mathbb{G}}(x_m,\delta_m)}\subset B_{\mathbb{G}}(x_{m-1},\delta_{m-1})\]
holds.
\end{lemma}


We record for later use that $\mathrm{Lip}_{\mathbb{G}}(f_m)\leq 1$ for all $m\geq 1$ and we define $\varepsilon'_m>0$ by letting
\begin{equation}\label{defepsprimo}
\varepsilon'_m:=\min\{\varepsilon_m/2,\, \sigma_{m-1}/2\}.
\end{equation}

We next show that the sets $D_{m}$ form a decreasing sequence. This is an adaptation of \cite[Lemma 3.3]{DM11}.

\begin{lemma}\label{lemmachiave}
The following statements hold.
\begin{enumerate}
\item If $m\geq 1$ and $(x,E)\in D_{m+1}$, then
\[(x_{m-1}, E_{m-1})\leq_{(f_m, \sigma_m-\varepsilon'_m)} (x,E).\]
\item If $m\geq 1$, then $D_{m+1}\subset D_m$.
\item If $m\geq 0$ and $(x,E)\in D_{m+1}$, then $ d(E(0),E_m(0))\leq \sigma_m$.
\end{enumerate}
\end{lemma}

\begin{proof}
If $m=0$ then (3) holds since
\[d(E(0),E_0(0))\leq d(E(0),0)+d(0,E_0(0))\leq 2=\sigma_0.\]
It is enough to check that, whenever $m\geq 1$ and (3) holds for $m-1$, then (1), (2) and (3) hold for $m$. Fix $m\geq 1$ and assume that (3) holds for $m-1$, i.e.
\[ d(E(0),E_{m-1}(0)) \leq \sigma_{m-1}\quad  \mbox{for all}\quad (x, E)\in D_m.\]

\medskip

\emph{Proof of (1).} Algorithm \ref{alg}(6) states that $(x_m,E_m)\in D_m$ and hence
\begin{equation}\label{stimavett}
d(E_m(0),E_{m-1}(0))\leq \sigma_{m-1}.
\end{equation}
Let $(x, E)\in D_{m+1}$. In particular, by Algorithm \ref{alg}(5) we have $Ef_{m+1}(x)\geq E_{m}f_{m+1}(x_{m})$. Notice that, since $\omega(E_{m})=\omega(E)=1$, we have $\langle E_m(0),E_m(0)\rangle=1$ and $\langle E(0), E_m(0) \rangle\leq 1$. Let $A:=Ef_m(x)-E_mf_m(x_m)$. Lemma \ref{lemmascalarlip} and the inequality $Ef_{m+1}(x)\geq E_{m}f_{m+1}(x_{m})$ give
\[Ef_{m+1}(x)-E_mf_{m+1}(x_m)-t_m\langle E(0), E_m(0) \rangle+t_m\geq 0.\]
Combining again Algorithm \ref{alg}(5) with the above inequality, gives
\[Ef_m(x)\geq E_mf_m(x_m)\geq E_{m-1} f_m(x_{m-1}).\]
In particular, $Ef_{m}(x)\geq E_{m-1}f_{m}(x_{m-1})$ which is the first requirement for $(x_{m-1}, E_{m-1})\leq_{(f_m, \sigma_m-\varepsilon'_m)} (x,E)$.

Let $B:=E f_m(x)-E_{m-1}f_m(x_{m-1})\geq 0$. Lemma \ref{lipismaximal} and $\mathrm{Lip}_{\mathbb{G}}(f_m)\leq 1$ implies that $0\leq A\leq B\leq 2$. Using these inequalities and $K\geq 4s^2$ gives
\begin{align}\label{factorize}
K(B^{\frac{1}{2s^2}}-A^{\frac{1}{2s^2}})&\geq (B^{\frac{2s^2-1}{2s^2}}+B^{\frac{2s^2-2}{2s^2}}A^{\frac{1}{2s^2}}+\ldots+B^{\frac{1}{2s^2}}A^{\frac{2s^2-2}{2s^2}}+A^{\frac{2s^2-1}{2s^2}})(B^{\frac{1}{2s^2}}-A^{\frac{1}{2s^2}})\nonumber  \\
&=B-A\nonumber \\
&=E_mf_m(x_m)-E_{m-1}f_m(x_{m-1}).
\end{align}
Since $A\geq Ef_{m+1}(x)-E_mf_{m+1}(x_m)$, \eqref{factorize} implies that
\begin{align}\label{estimate3}
&  E_mf_m(x_m)-E_{m-1}f_m(x_{m-1})+K( Ef_{m+1}(x)-E_mf_{m+1}(x_m))^{\frac{1}{2s^2}}\nonumber \\
& \qquad \leq K B^{\frac{1}{2s^2}}.
\end{align}
To prove the second requirement of $(x_{m-1}, E_{m-1})\leq_{(f_m, \sigma_m-\varepsilon'_m)} (x,E)$ we need to estimate
\begin{equation}\label{thingtoestimate}
|(f_m(x\exp(tE_{m-1}))-f_m(x))-(f_m(x_{m-1}\exp(tE_{m-1}))-f_m(x_{m-1}))|.
\end{equation}
We consider two cases, depending on whether $t$ is small or large. 

\medskip

\emph{Suppose $|t|<3C_D\delta_m^{\frac{1}{s}}/\varepsilon_m$.}
Estimate \eqref{thingtoestimate} as follows
\begin{align}\label{estimate}
&|(f_m(x\exp(tE_{m-1})) - f_m(x))-(f_m(x_{m-1}\exp(tE_{m-1}))-f_m(x_{m-1}))|\nonumber \\
&\qquad \leq |(f_m(x\exp(tE_m))-f_m(x)) - (f_m(x_m\exp(tE_m))-f_m(x_m))| \nonumber \\
&\qquad \quad + |(f_m(x_m\exp(tE_m))-f_m(x_m))\nonumber \\
&\qquad \quad \qquad -(f_m(x_{m-1}\exp(tE_{m-1}))-f_m(x_{m-1}))|  \nonumber \\
&\qquad \quad +| f_m(x\exp(tE_{m-1}))-f_m(x\exp(tE_m))|.
\end{align}
We consider the three terms on the right hand side of \eqref{estimate} separately.

Firstly, Algorithm \ref{alg}(1) and Lemma \ref{lemmascalarlip} give
\begin{align}\label{eqz1}
&(f_m(x\exp(tE_m))-f_m(x))-(f_m (x_{m}\exp(tE_{m}))- f_m(x_{m}))\\
&\qquad =(f_{m+1}(x\exp(tE_m))-f_{m+1}(x)) - ( f_{m+1}(x_m\exp(tE_m)) - f_{m+1}(x_m)) \nonumber \\
&\qquad \quad  -t_m\langle x\exp(tE_m), E_m(0)\rangle +t_m\langle x,E_m(0)\rangle \nonumber \\
&\qquad \quad+t_m\langle x_m\exp(tE_m), E_m(0)\rangle- t_m\langle x_m, E_m(0)\rangle \nonumber \\
&\qquad=(f_{m+1}(x\exp(tE_m))-f_{m+1}(x)) - ( f_{m+1}(x_m\exp(tE_m)) - f_{m+1}(x_m))\nonumber .
\end{align}
Since $(x,E)\in D_{m+1}$, using \eqref{eqz1} gives
\begin{align}\label{estimate2}
&|(f_m(x\exp(tE_{m}))-f_m(x))-(f_m(x_{m}\exp(tE_{m}))-f_m(x_{m}))| \nonumber \\
&\qquad \leq K(\sigma_m+(Ef_{m+1}(x)-E_mf_{m+1}(x_m))^{\frac{1}{2s^2}})|t|.
\end{align}

For the second term in \eqref{estimate} we recall that, for the values of $t$ we are considering, Algorithm \ref{alg}(8) states that
\begin{align}
&|(f_m(x_m\exp(tE_m))-f_m(x_m))-(f_m(x_{m-1}\exp(tE_{m-1}))-f_m(x_{m-1}))| \label{feb22a}\\
&\qquad \leq( E_mf_m(x_m)-E_{m-1}f_m(x_{m-1})+\sigma_{m-1}) |t|.\nonumber
\end{align}

The final term in \eqref{estimate} is estimated using $\mathrm{Lip}_{\mathbb{G}}(f_{m})\leq 1$ and \eqref{stimavett}:
\begin{align}
|f_m(x\exp(tE_{m-1}))-f_m(x\exp(tE_m))| &\leq d(x\exp(tE_{m-1}), x\exp(tE_m)) \label{feb22b}\\
&= d(tE_{m-1}(0),tE_{m}(0)) \nonumber \\
& \leq \sigma_{m-1}|t|.\nonumber 
\end{align}

Adding \eqref{estimate2}, \eqref{feb22a} and \eqref{feb22b}, then using \eqref{estimate3}, \eqref{defepsprimo} and Algorithm \ref{alg}(2), gives
\begin{align*}
&|(f_m(x\exp(tE_{m-1}))-f_m(x))-(f_m(x_{m-1}\exp(tE_{m-1}))-f_m(x_{m-1}))| \\
& \qquad \leq K(\sigma_m + (Ef_{m+1}(x)-E_mf_{m+1}(x_m))^{\frac{1}{2s^2}})|t| \\
&\qquad \quad + ( E_mf_m(x_m)-E_{m-1}f_m(x_{m-1})+\sigma_{m-1}) |t|\\
&\qquad \quad + \sigma_{m-1}|t|\\
& \qquad \leq K(\sigma_{m-1}-\varepsilon_{m}' + (Ef_{m}(x)-E_{m-1}f_{m}(x_{m-1}))^{\frac{1}{2s^2}})|t|,
\end{align*}
which gives the required estimate of \eqref{thingtoestimate} for all $t$ with $|t|<3C_D\delta_m^{\frac{1}{s}}/\varepsilon_m$.

\emph{Suppose $3C_D\delta_m^{\frac{1}{s}}/\varepsilon_m \leq |t| < 1$.}
In particular, this implies
\begin{equation}\label{refest} \delta_{m} \leq \varepsilon_{m}^{s}t^{s}/3C_D \leq \varepsilon_{m}|t|,\end{equation}
where in the last inequality above we used that
\[\varepsilon_{m}|t|/3C_D\leq \varepsilon_{m}/3C_D\leq 1,\]
which follows from $\varepsilon_{m}\leq 2$ and $C_D\geq 1$.

We estimate \eqref{thingtoestimate} as follows:
\begin{align*}
&|(f_m(x\exp(tE_{m-1}))-f_m(x))-(f_m(x_{m-1}\exp(tE_{m-1}))-f_m(x_{m-1}))|\\
&\qquad \leq |(f_m(x_m\exp(tE_{m-1}))-f_m(x_m))\\
&\qquad \quad \quad -(f_m(x_{m-1}\exp(tE_{m-1}))-f_m(x_{m-1}))|\\
&\qquad \quad + |f_{m}(x) - f_{m}(x_{m})|\\
&\qquad \quad + |f_m(x\exp(tE_{m-1})) - f_m(x_m\exp(tE_{m-1}))|,
\end{align*}
and again we separately consider the three terms on the right hand side.

By Algorithm \ref{alg}(7) we have
\[(x_{m-1}, E_{m-1})\leq_{(f_m, \sigma_{m-1}-\varepsilon_m)} (x_m, E_m),\] 
which gives
\begin{align}
&|(f_m(x_m\exp(tE_{m-1}))-f_m(x_m))-(f_m(x_{m-1}\exp(tE_{m-1}))-f_m(x_{m-1}))|\nonumber \\
&\qquad \leq K(\sigma_{m-1}-\varepsilon_{m} + (E_{m}f_{m}(x_{m})-E_{m-1}f_{m}(x_{m-1}))^{\frac{1}{2s^2}})|t|.\label{feb221}
\end{align}

For the estimate of the second term we use $\mathrm{Lip}_{\mathbb{G}}(f_{m})\leq 1$ and \eqref{refest} to get
\begin{equation}\label{feb222}
|f_m(x)-f_m(x_m)|\leq d(x,x_{m}) \leq \delta_m\leq \varepsilon_m |t|\leq K\varepsilon_m |t|/(4s^2).
\end{equation}

Notice that $x\exp(tE_{m-1})$ and  $x_{m}\exp(tE_{m-1})$ belong to $\overline{B_{\mathbb{G}}(x_{0},2+\delta_{0})}$. 
Using Proposition \ref{euclideanheisenberg} and Proposition \ref{conjugatedistance}, recalling that $\delta_m <1$, we get
\begin{align}
&|f_m(x\exp(tE_{m-1}))-f_m(x_m\exp(tE_{m-1}))|\nonumber \\
&\qquad \leq d(x\exp(tE_{m-1}),x_{m}\exp(tE_{m-1}))\nonumber \\
&\qquad = d(\exp(tE_{m-1})^{-1}x_m^{-1} x\exp(tE_{m-1}))\nonumber \\
&\qquad \leq C_D(d(x_m, x)+ t^{\frac{1}{s}} d(x_m, x)^{\frac{s-1}{s}}+t^{\frac{s-1}{s}} d(x_m, x)^{\frac{1}{s}})\nonumber \\
&\qquad \leq C_D(\delta_m+ t^{\frac{1}{s}} \delta_m^{\frac{s-1}{s}}+t^{\frac{s-1}{s}} \delta_m^{\frac{1}{s}}) \nonumber \\
&\qquad \leq 3C_D \delta_m^{\frac{1}{s}}
 \leq \varepsilon_{m}|t| 
\leq K\varepsilon_{m} |t|/(4s^2). \label{feb223}
\end{align}

Combine \eqref{feb221}, \eqref{feb222} and \eqref{feb223} to obtain
\begin{align*}
&|(f_m(x\exp(tE_{m-1}))-f_m(x))-(f_m(x_{m-1}\exp(tE_{m-1}))-f_m(x_{m-1}))|\\
&\qquad \leq K(\sigma_{m-1}-\varepsilon_m/2+(E_mf_m(x_m)-E_{m-1}f_m(x_{m-1}))^{\frac{1}{2s^2}})\\
&\qquad \leq K(\sigma_{m-1}-\varepsilon'_m+(Ef_m(x)-E_{m-1}f_m(x_{m-1}))^{\frac{1}{2s^2}}),
\end{align*}
which gives the required estimate of \eqref{thingtoestimate} for all $t$ satisfying $3C_D\delta_m^{\frac{1}{s}}/\varepsilon_m \leq |t| < 1$. This completes the proof of (1).

\medskip

\emph{Proof of (2).} Suppose $(x,E)\in D_{m+1}$. Then $(x,E)\in D^{f_{m+1}}=D^{f_{m}}$ and Lemma \ref{inclusionballs} implies that $d(x,x_{m-1})<\delta_{m-1}$.
 Combining this with (1) gives $(x,E)\in D_{m}$. This completes the proof of (2).

\medskip
 
\emph{Proof of (3).} Suppose $(x,E)\in D_{m+1}$. Then by Algorithm \ref{alg}(5) we have $E_m f_{m+1}(x_m)\leq E f_{m+1}(x)$. Moreover, by Algorithm \ref{alg}(1) we have
\begin{equation}\label{aaa}
E_m f_{m}(x_m)+t_m\langle E_{m}(0),E_{m}(0) \rangle \leq E f_{m}(x)+t_m\langle E(0),E_{m}(0) \rangle.
\end{equation}
By (2), we have also that $(x,E)\in D_m$, so Algorithm \ref{alg}(6) implies
\begin{equation}\label{bbb}
Ef_m(x)\leq E_m f_m(x_m)+\lambda_m.
\end{equation}
Combining \eqref{aaa} and \eqref{bbb} gives $t_m\leq t_m\langle E(0),E_m(0)\rangle+\lambda_m$, which, up to rearrangements, implies
\[\langle E(0),E_{m}(0) \rangle\geq 1-\lambda_m/t_m.\]
Therefore one has
\begin{equation}\label{stimam}
|p(E)-p(E_m)|=|E(0)-E_m(0)|=(2-2\langle E(0), E_m(0))^{\frac{1}{2}}\leq (2\lambda_m/t_m)^{\frac{1}{2}}.
\end{equation}
Combining Algorithm \ref{alg}(5) and Lemma \ref{closedirectioncloseposition} with $g(t):=\exp(tE_m)$, $N=1$ and $D=(2\lambda_m/t_m)^{\frac{1}{2}}$, we get
\begin{align*}
d(E(0),E_m(0))=d(\exp(E),\exp(E_m)) &\leq C_{\mathrm{a}}(2\lambda_m/t_m)^{\frac{1}{2s^2}}\\
& \leq \sigma_m,
\end{align*}
which proves (3).
\end{proof}

We next study the convergence of $(x_{m}, E_{m})$ and $f_{m}$. We show that the directional derivatives converge to a directional derivative of the limiting function, and the limit of $(x_{m},E_{m})$ belongs to $D_{m}$ for every $m$. This is an adaptation of \cite[Lemma 3.4]{DM11}.

\begin{lemma}\label{lemmaquasifinale}
The following statements hold:
\begin{enumerate}
\item $f_{m}\to f$ pointwise, where $f:\mathbb{G}\to \mathbb{R}$ is Lipschitz and $\mathrm{Lip}_{\mathbb{G}}(f)\leq 1$,
\item $f-f_m$ is $\mathbb{G}$-linear and $\mathrm{Lip}_{\mathbb{G}}(f-f_m) \leq 2t_m$ for $m\geq 0$,
\item There exist $x_{\ast}\in N$ and $E_{\ast} \in V$ with $\omega(E_{\ast})=1$ such that for $m \geq 0$ we have
\[d(x_{\ast},x_m)< \delta_m, \quad\mbox{and}\quad  d(E_{\ast}(0),E_m(0))\leq \sigma_m.\]
\item $E_{\ast}f(x_{\ast})$ exists, is strictly positive and $E_mf_m(x_m)\uparrow E_{\ast}f(x_{\ast})$,
\item $(x_{m-1},E_{m-1})\leq_{(f_m, \sigma_{m-1}-\varepsilon'_m)} (x_{\ast},E_{\ast})$ for $m\geq 1$,
\item $(x_{\ast},E_{\ast})\in D_m$ for $m\geq 1$,
\item $\omega(E_{\ast}-E_0)<\tau$.
\end{enumerate}
\end{lemma}

\begin{proof}
We prove each statement individually.

\medskip

\emph{Proof of (1).} Algorithm \ref{alg}(1) gives $f_m(x)=f_0(x)+\langle x,\sum_{k=0}^{m-1}t_kE_{k}(0)\rangle$. Define $f:\mathbb{G}\to\mathbb{R}$ by
\begin{equation}\label{deff}
f(x):=f_0(x)+\Big\langle x,\sum_{k=0}^{\infty}t_k E_{k}(0)\Big\rangle.
\end{equation}
Notice $|f(x)-f_m(x)|\leq | x | \sum_{k=m}^{\infty} t_k |E_{k}(0)|$. Hence by Algorithm \ref{alg}(3) $f_m\to f$ pointwise and, since $\mathrm{Lip}_{\mathbb{G}}(f_m)\leq 1$, we deduce $\mathrm{Lip}_{\mathbb{G}}(f)\leq 1$.

\medskip

\emph{Proof of (2).} Lemma \ref{lemmascalarlip} shows that $f-f_{m}$ is $\mathbb{G}$-linear. Moreover, by Algorithm \ref{alg}(3), we have that for every $m\geq 0$
\[\mathrm{Lip}_{\mathbb{G}}(f-f_m)  \leq \sum_{k=m}^{\infty} t_k \leq t_m \sum_{k=m}^{\infty} \frac{1}{2^{k-m}} \leq 2t_m.\]

\medskip

\emph{Proof of (3).} Let $q\geq m\geq 0$. The definition of $D_{q+1}$ in Algorithm \ref{alg}(5) shows that $(x_q,E_q)\in D_{q+1}$. Hence points 2 and 3 of Lemma \ref{lemmachiave} imply that $(x_q,E_q)\in D_{m+1}$, and consequently
\begin{equation}\label{Cauchy1}
d(E_q(0),E_m(0)) \leq \sigma_m.
\end{equation}
Since $(x_q,E_q)\in D_{m+1}$, Algorithm \ref{alg}(5) implies
\begin{equation}\label{Cauchy2}
d(x_q,x_m)< \delta_m.
\end{equation}

Since, by Lemma \ref{inclusionballs}, $\sigma_m, \delta_m \to 0$ the sequences $(x_m)_{m=1}^{\infty}$ and $(E_{m}(0))_{m=1}^{\infty}$ are Cauchy, and therefore they converge to some $x_{\ast}\in\mathbb{G}$ and $v\in \mathbb{G}$, respectively. Since $E_{m}\in V_1$ and $\omega(E_{m})=1$, we know that $|p(v)|=1$ and $v=(p(v),0)$. Using group translations, we can extend $v$ to a vector field $E_{\ast}\in V_1$ with $\omega(E_{\ast})=1$ and $E_{\ast}(0)=v$. Letting $q\to \infty$ in \eqref{Cauchy1} and \eqref{Cauchy2} implies that $d(E_{\ast}(0),E_m(0)) \leq \sigma_m$ and $d(x_{\ast},x_m)\leq \delta_m$. Combining Lemma \ref{inclusionballs} and the fact that $\delta_m<\delta_{m-1}/2$, we have the strict inequality $d(x_{\ast},x_m)< \delta_m$.

We now know that $x_{\ast}\in \overline{B_{\mathbb{G}}(x_m,\delta_m)}$ for every $m\geq 1$. Recall that $N=\cap_{m=0}^{\infty} U_m$ for open sets $U_m \subset \mathbb{G}$, and Algorithm \ref{alg}(8) states that $\overline{B_{\mathbb{G}}(x_{m},\delta_{m})}\subset U_{m}$. Hence $x_{\ast}\in N$. 


\medskip

\emph{Proof of (4).} As in the proof of (3) we have $(x_q,E_q)\in D_{m+1}$ for every $q\geq m\geq 0$. Therefore, by Lemma \ref{lemmachiave}(1), for every $q\geq m\geq 1$ we have
\begin{equation}\label{bla}
(x_{m-1}, E_{m-1})\leq_{(f_m,\sigma_{m-1}-\varepsilon'_m)} (x_q,E_q).
\end{equation}
Algorithm \ref{alg}(1) and \eqref{bla} (with $m$ and $q$ replaced by $q+1$) give
\begin{equation}\label{stima} 
E_qf_q(x_q)< E_q f_{q+1}(x_q)\leq E_{q+1}f_{q+1}(x_{q+1}) \quad \mbox{for every}\quad q \geq 0.
\end{equation} 
Hence, since $E_0f_0(x_0)\geq 0$, the sequence $(E_q f_q(x_q))_{q=0}^{\infty}$ is strictly increasing and positive.

Since $\mathrm{Lip}_{\mathbb{G}}(f_q)\leq 1$ for every $q\geq 1$, by Lemma \ref{lipismaximal}, the sequence $(E_qf_q(x_q))_{q=1}^{\infty}$ is bounded above by $1$. Consequently, $E_qf_q(x_q)\to L$ for some $0<L\leq 1$. Inequality \eqref{stima} implies that also $E_q f_{q+1}(x_q) \to L$, and, moreover, one has
\[E_q f(x_q)=E_q f_q(x_q)+E_q (f-f_q)(x_q)\]
and $|E_q (f-f_q)(x_q)|\leq \mathrm{Lip}_{\mathbb{G}}(f-f_q) \leq 2t_{q} \to 0$. Hence also $E_qf(x_q) \to L$.

Let $q\geq m\geq 0$ and consider
\[s_{m,q}:=E_qf_m(x_q)-E_{m-1}f_m(x_{m-1}).\]
By \eqref{bla} we have that $s_{m,q}\geq 0$. Letting $q\to \infty$, writing $f_{m}=f+(f_{m}-f)$, and using the $\mathbb{G}$-linearity of $f_{m}-f$ one gets
\begin{equation}\label{defiC}
s_{m,q}\to s_{m}:=(f_m-f)(E_{\ast}(0))+L-E_{m-1}f_m(x_{m-1})\geq 0.
\end{equation}
Since $\mathrm{Lip}_{\mathbb{G}}(f_{m}-f)\leq 2t_{m}$ and $E_{m-1}f_{m}(x_{m-1})\to L$, also $s_{m} \to 0$ as $m\to \infty$.
\eqref{bla} implies that
\begin{align}\label{bla2}
&|(f_m(x_q\exp(tE_{m-1})) - f_m(x_q))-(f_m(x_{m-1}\exp(tE_{m-1}))-f_m(x_{m-1}))|\nonumber \\
&\qquad \leq K (\sigma_{m-1}-\varepsilon'_m+(s_{m,q})^{\frac{1}{2s^2}}) |t| \quad \mbox{ for }t\in (-1,1).
\end{align}
Letting $q\to \infty$ in \eqref{bla2} shows that
\begin{align}\label{eqncruc}
&|(f_m(x_{\ast}\exp(tE_{m-1}))-f_m(x_{\ast}))-(f_m(x_{m-1}\exp(tE_{m-1}))-f_m(x_{m-1}))|\nonumber \\
&\qquad \leq K(\sigma_{m-1}-\varepsilon'_m+(s_{m})^{\frac{1}{2s^2}})|t| \quad \mbox{ for }t\in (-1,1).
\end{align}
Since $\mathrm{Lip}_{\mathbb{G}}(f)\leq 1$ and $d(E_{\ast}(0),E_{m-1}(0))\leq \sigma_{m-1}$, we obtain
\begin{align}
|f(x_{\ast}\exp(tE_{\ast}))-f(x_{\ast}\exp(tE_{m-1}))|&\leq d(x_{\ast}(tE_{\ast}(0)),x_{\ast}(tE_{m-1}(0)))\nonumber \\
&\leq \sigma_{m-1}|t|. \label{yyy}
\end{align}
Since $f-f_{m}$ is $\mathbb{G}$-linear and $\mathrm{Lip}_{\mathbb{G}}(f-f_m) \leq 2t_m$ we can estimate
\begin{align}
|(f-f_m)(x_{\ast}\exp(tE_{m-1}))-(f-f_m)(x_{\ast})| &= |(f-f_{m})(\exp(tE_{m-1}))|\nonumber \\
&= |(f-f_{m})(\delta_t(\exp(E_{m-1})))|\nonumber \\
&\leq t\mathrm{Lip}_{\mathbb{G}}(f-f_{m})\nonumber \\
&\leq 2t_{m}|t|.\label{zzz}
\end{align}

Combining \eqref{eqncruc}, \eqref{yyy} and \eqref{zzz} shows that for $t\in (-1,1)$:
\begin{align*}
&|(f(x_{\ast}\exp(tE_{\ast}))-f(x_{\ast}))-(f_m(x_{m-1}\exp(tE_{m-1}))-f_m(x_{m-1}))|\\
&\qquad \leq |(f_m(x_{\ast}\exp(tE_{m-1}))-f_m(x_{\ast}))\\
&\qquad \quad \quad -(f_m(x_{m-1}\exp(tE_{m-1}))-f_m(x_{m-1}))|\\
&\qquad \quad +|f(x_{\ast}\exp(tE_{\ast}))-f(x_{\ast}\exp(tE_{m-1}))|\\
&\qquad \quad +|(f-f_m)(x_{\ast}\exp(tE_{m-1}))-(f-f_m)(x_{\ast})|\\
&\qquad \leq (K(\sigma_{m-1}-\varepsilon'_m+(s_{m})^{\frac{1}{2s^2}})+\sigma_{m-1}+2t_m)|t|.
\end{align*}

Fix $\varepsilon>0$ and choose $m\geq 1$ such that
\[K(\sigma_{m-1}-\varepsilon'_m+(s_{m})^{\frac{1}{2s^2}})+\sigma_{m-1}+2t_m\leq \varepsilon/3\]
and
\[|E_{m-1}f_m(x_{m-1})-L|\leq \varepsilon/3.\]
Using the definition of $E_{m-1}f_{m}(x_{m-1})$, we find $0<\delta<1$ such that for every $|t|< \delta$
\[|f_m(x_{m-1}\exp(tE_{m-1}))-f_m(x_{m-1})-tE_{m-1}f_m(x_{m-1})|\leq \varepsilon|t|/3.\]
Hence, for every $|t|< \delta$
\begin{align*}
&|f(x_{\ast}\exp(tE_{\ast}))-f(x_{\ast})-tL|\\
&\qquad \leq|(f(x_{\ast}\exp(tE_{\ast}))-f(x_{\ast}))-(f_m(x_{m-1}\exp(tE_{m-1}))-f_m(x_{m-1}))|\\
&\qquad \quad +|f_m(x_{m-1}\exp(tE_{m-1}))-f_m(x_{m-1})-tE_{m-1}f_m(x_{m-1})|\\
&\qquad \quad +|E_{m-1}f_m(x_{m-1})-L| |t|\\
&\qquad \leq \varepsilon |t|.
\end{align*}
This proves that $E_{\ast}f(x_{\ast})$ exists and is equal to $L$. We have already seen that $(E_qf_q(x_q))_{q=1}^{\infty}$ is a strictly increasing sequence of positive numbers. This proves (4).

\medskip

\emph{Proof of (5).} The definition of $L$ and Lemma \ref{lemmascalarlip} imply
\[E_{\ast}f_m(x_{\ast})=L+E_{\ast}(f_{m}-f)(x_{\ast})=L+(f_m-f)(E_{\ast}(0)).\]
Using \eqref{defiC} shows $s_{m}=E_{\ast}f_m(x_{\ast})-E_{m-1}f_m(x_{m-1})\geq 0$. Substituting this in \eqref{eqncruc} gives (5).

\medskip

\emph{Proof of (6).} Property (6) is a consequence of (3), (4) and (5).

\medskip

\emph{Proof of (7).} We start by estimating $\omega(E_1-E_0)$.
By Algorithm \ref{alg}(6), for every $(x,E)\in D_1$ we have
\begin{equation*}Ef_1(x)\leq E_1f_1(x_1)+\lambda_1,\end{equation*}
where 
\begin{align}\label{defF}
f_1(x)=f_0(x)+t_0\left\langle x, E_0(0)\right\rangle.
\end{align}
Clearly
\[E_{0}f_{1}(x_{0})=E_{0}f_{0}(x_{0})+t_{0}\]
and
\[E_1 f_1(x_1)=E_{1}f_{0}(x_{1})+t_{0}\langle E_1(0), E_0(0)\rangle.\]
By Algorithm \ref{alg}(5), $(x_0,E_0)\in D_1$ and therefore
\begin{align}\label{disE}
E_0f_1(x_0)\leq E_1f_1(x_1)+\lambda_1.
\end{align}
A simple calculation using \eqref{disE} then gives
\begin{equation}
\omega(E_1,E_0)=\left\langle E_1(0), E_0(0)\right\rangle\geq 1-\frac{\lambda_1}{t_0}+\frac{E_0f_0(x_0)-E_1f_0(x_1)}{t_0}.
\end{equation}
Since $\omega$ is an inner product norm we can estimate as follows:
\begin{align*}
\omega(E_1-E_0)&=\left(\omega(E_1)^2+\omega(E_0)^2-2\omega(E_1,E_0)\right)^{\frac{1}{2}}\\
&=\left(2-2\omega(E_1,E_0)\right)^{\frac{1}{2}}\\
&\leq \left(\frac{2\lambda_1}{t_0}+\frac{2|E_0f_0(x_0)-E_1f_0(x_1)|}{t_0}\right)^{\frac{1}{2}}\\
&\leq \left(\frac{2\lambda_1}{t_0}+\frac{4\mathrm{Lip}_{\bbG}(f_0)}{t_0}\right)^{\frac{1}{2}}\\
&\leq \frac{\tau}{2},
\end{align*}
where in the last inequality above we used the estimate on $\mathrm{Lip}_{\bbG}(f_0)$ in Assumption \ref{Ass} and the estimate on $\lambda_{1}$ in Algorithm \ref{alg}(3). 
Next, as proved in \eqref{stimam}, for every $m\geq 1$ and $(x, E)\in D_{m+1}$ we have
\begin{equation}
\omega(E-E_{m})\leq \left(2\lambda_m/t_m\right)^{\frac{1}{2}}.
\end{equation}
Using the estimate in Algorithm \ref{alg}(4), this implies that for every $m\geq 1$:
\begin{align}
\omega(E_{m+1}-E_{m})< \frac{\tau}{2^{m+1}}.
\end{align}
Therefore,
\begin{align*}
\omega(E_{\ast}-E_0)=\lim_{m\to \infty} \omega(E_m-E_0)&\leq \sum_{m=2}^{\infty}\omega(E_m-E_{m-1})+\omega(E_1-E_0)\\
&< \tau \sum_{m=2}^{\infty}\frac{1}{2^{m}}+ \frac{\tau}{2}\\
&=\tau.
\end{align*}
\end{proof}

We now prove that the limit directional derivative $E_{\ast}f(x_{\ast})$ is almost locally maximal in horizontal directions. This is an adaptation of \cite[Lemma 3.5]{DM11}. 

\begin{lemma}\label{almostlocmax}
For all $\varepsilon>0$ there is $\delta_{\varepsilon}>0$ such that if $(x,E)\in D^f$ satisfies $d(x_{\ast},x)\leq \delta_{\varepsilon}$ and $(x_{\ast},E_{\ast})\leq_{(f,0)}(x,E)$, then
\[Ef(x)<E_{\ast}f(x_{\ast})+\varepsilon.\]
\end{lemma}

\begin{proof}
Fix $\varepsilon>0$. By Lemma \ref{inclusionballs} we choose $m\geq 1$ such that
\begin{equation}\label{param} 
m\geq 4/\varepsilon^{\frac{2s^{2}-1}{2s^{2}}}\quad \mbox{and}\quad \lambda_m,t_m\leq \varepsilon/4.
\end{equation}
Recall that $\varepsilon'_m=\min\{\varepsilon_m/2,\, \sigma_{m-1}/2\}$. Using Lemma \ref{lemmaquasifinale}(3) and \ref{lemmaquasifinale}(6) , fix $\delta_{\varepsilon}>0$ such that 
\[\delta_{\varepsilon}< \delta_{m-1}-d(x_{\ast},x_{m-1}) 
\]
such that for every $|t|< 3D\delta_{\varepsilon}^{\frac{1}{s}}/\varepsilon_{m}'$
\begin{align}\label{estimated2}
&|(f_m(x_{\ast}\exp(tE_{\ast}))-f_m(x_{\ast}))-(f_m(x_{m-1}\exp(tE_{m-1}))-f_m(x_{m-1}))|\nonumber \\
&\qquad \leq (E_{\ast}f_m(x_{\ast})-E_{m-1}f_m(x_{m-1})+\sigma_{m-1})|t|.
\end{align}
Such $\delta_{\varepsilon}$ exists since, by Lemma \ref{lemmaquasifinale}(5), we have $E_{\ast}f_m(x_{\ast})\geq E_{m-1}f_m(x_{m-1})$. 

We argue by contradiction and we suppose that $(x,E)\in D^f$ satisfies $d(x_{\ast},x)\leq \delta_{\varepsilon}$, $(x_{\ast},E_{\ast})\leq_{(f,0)} (x,E)$ and $Ef(x)\geq E_{\ast}f(x_{\ast})+\varepsilon$. We plan to show that $(x,E)\in D_m$. We first observe that this gives a contradiction. Indeed, Algorithm \ref{alg}(6) and the monotone convergence $E_mf_m(x_m)\uparrow E_{\ast}f(x_{\ast})$ would then imply
\[Ef_m(x)\leq E_mf_m(x_m)+\lambda_m\leq E_{\ast}f(x_{\ast})+\lambda_m.\]
By Lemma \ref{lemmaquasifinale}(2) and \eqref{param} we would deduce that
\begin{align*}
Ef(x)-E_{\ast}f(x_{\ast})&=(Ef_m(x)-E_{\ast}f(x_{\ast}))+E(f-f_m)(x)\\
&\leq \lambda_m+2t_m\\
&\leq 3\varepsilon /4,
\end{align*}
which contradicts the assumption that $Ef(x)\geq E_{\ast}f(x_{\ast})+\varepsilon$.

\medskip

\emph{Proof that $(x,E)\in D_m$.} Since $f-f_{m}$ is $\mathbb{G}$-linear we have  $D^f=D^{f_{m}}$ and therefore $(x,E)\in D^{f_{m}}$. Next observe that
\[d(x,x_{m-1})\leq d(x,x_{\ast})+d(x_{\ast},x_{m-1}) < \delta_{m-1}.\]
Hence, it suffices to show that $(x_{m-1},E_{m-1})\leq_{(f_m, \sigma_{m-1}-\varepsilon'_m/2)} (x,E)$.
Lemma \ref{lipismaximal} implies
\[|E(f-f_{m})(x)|,\, |E_{\ast}(f-f_{m})(x_{\ast})|\leq \mathrm{Lip}_{\mathbb{G}}(f-f_{m}).\]
Hence by definition of $(x,E)$ and by \eqref{param} we have
\begin{align*}
Ef_m(x)-E_{\ast}f_m(x_{\ast})&\geq Ef(x)-E_{\ast}f(x_{\ast})-2\mathrm{Lip}_{\mathbb{G}}(f_m-f)\\
&\geq \varepsilon -4t_m\geq 0.
\end{align*}
Lemma \ref{lemmaquasifinale}(6) states that $(x_{\ast}, E_{\ast})\in D_{m}$, which implies $E_{m-1}f_m(x_{m-1})\leq E_{\ast}f_m(x_{\ast})$ and hence
\[Ef_m(x)\geq E_{\ast}f_m(x_{\ast})\geq E_{m-1}f_m(x_{m-1}).\]
In particular, the inequality $Ef_{m}(x)\geq E_{m-1}f_m(x_{m-1})$ proves the first requirement of $(x_{m-1},E_{m-1})\leq_{(f_m, \sigma_{m-1}-\varepsilon'_m/2)} (x,E)$.

We next deduce several inequalities from our hypotheses. Let
\begin{itemize}
\item $A:=Ef(x)-E_{\ast}f(x_{\ast})$,
\item $B:=Ef_m(x)-E_{\ast}f_m(x_{\ast})$,
\item $C:=Ef_m(x)-E_{m-1}f_m(x_{m-1})$.
\end{itemize}
By definition of $(x,E)$ we have $A\geq \varepsilon$, while the inequalities above give $0\leq B\leq C$. By Lemma \ref{lipismaximal} we have that $A,\, B,\, C\leq 2$. Recalling the factorization
\begin{equation}\label{factorizerepeat}
A-B=\left(B^{\frac{2s^2-1}{2s^2}}+B^{\frac{s^2-1}{s^2}}A^{\frac{1}{2s^2}}+\ldots+B^{\frac{1}{2s^2}}A^{\frac{s^2-1}{s^2}}+A^{\frac{2s^2-1}{2s^2}}\right)\left(A^{\frac{1}{2s^2}}-B^{\frac{1}{2s^2}}\right)
\end{equation}
and using Lemma \ref{lemmaquasifinale}(2), \eqref{param} and Algorithm \ref{alg}(3), we obtain
\begin{align*}
A^{\frac{1}{2s^2}}-B^{\frac{1}{2s^2}} &\leq (A-B)/\varepsilon^{\frac{2s^2-1}{2s^2}}\\
&=(E(f-f_m)(x)-E_{\ast}(f-f_m)(x_{\ast}))/\varepsilon^{\frac{2s^2-1}{2s^2}}\\\
&\leq 4t_m /\varepsilon^{\frac{2s^2-1}{2s^2}}\\\
& \leq mt_m\\
&\leq \sigma_{m-1}/s^2.
\end{align*}
Since $B,\, C\leq 2$ and $K\geq 4s^2$ we have
\[B^{\frac{2s^2-1}{2s^2}}+B^{\frac{s^2-1}{s^2}}A^{\frac{1}{2s^2}}+\ldots+B^{\frac{1}{2s^2}}A^{\frac{s^2-1}{s^2}}+A^{\frac{2s^2-1}{2s^2}}\leq 4s^2 \leq K.\]
Hence using \eqref{factorizerepeat} with $A$ replaced by $C$ gives
\[KC^{\frac{1}{2s^2}}-KB^{\frac{1}{2s^2}}\geq C-B=E_{\ast}f_m(x_{\ast})-E_{m-1}f_m(x_{m-1}).\]
Combining our estimates we eventually find
\begin{align}\label{stima32}
&E_{\ast}f_m(x_{\ast})-E_{m-1}f_m(x_{m-1})+K(Ef(x)-E_{\ast}f(x_{\ast}))^{\frac{1}{2s^2}}\nonumber \\
&\qquad =E_{\ast}f_m(x_{\ast})-E_{m-1}f_m(x_{m-1})+KA^{\frac{1}{2s^2}}\nonumber\\
&\qquad \leq KC^{\frac{1}{2s^2}}-KB^{\frac{1}{2s^2}}+K(B^{\frac{1}{2s^2}}+\sigma_{m-1}/s^2) \nonumber\\
&\qquad = K((Ef_m(x)-E_{m-1}f_m(x_{m-1}))^{\frac{1}{2s^2}}+\sigma_{m-1}/s^2).
\end{align}

We now prove the second requirement of $(x_{m-1},E_{m-1})\leq_{(f_m, \sigma_{m-1}-\varepsilon'_m/2)} (x,E)$. We need to estimate
\begin{equation}\label{incases}
|(f_m(x\exp(tE_{m-1}))-f_m(x))-(f_m(x_{m-1}\exp(tE_{m-1}))-f_m(x_{m-1}))|.
\end{equation}

We consider two cases, depending on whether $t$ is small or large.

\medskip

\emph{Suppose $|t|\leq 3C_D\delta_{\varepsilon}^{\frac{1}{s}}/\varepsilon_{m}'$.}
To estimate \eqref{incases} we use the inequality
\begin{align}\label{toestimate}
&|(f_m(x\exp(tE_{m-1})) - f_m(x))-(f_m(x_{m-1}\exp(tE_{m-1})) - f_m(x_{m-1})|\nonumber \\
&\qquad \leq |(f_m(x\exp(tE_{\ast})) - f_m(x))-(f_m(x_{\ast}\exp(tE_{\ast})) - f_m(x_{\ast}))| \nonumber \\
&\qquad \quad + |(f_m(x_{\ast}\exp(tE_{\ast})) - f_m(x_{\ast}))\nonumber \\
&\qquad \quad \qquad -(f_m(x_{m-1}\exp(tE_{m-1}))-f_m(x_{m-1}))| \nonumber \\
&\qquad \quad +|f_m(x\exp(tE_{m-1}))-f_m(x\exp(tE_{\ast})|.
\end{align}
Since $(x_{\ast},E_{\ast})\leq_{(f,0)} (x,E)$, by Lemma \ref{lemmascalarlip} and $\mathbb{G}$-linearity of $f_{m}-f$ we can estimate the first term in \eqref{toestimate} by
\begin{align}\label{estimated1}
&|(f_m(x\exp(tE_{\ast}))-f_m(x))-(f_m(x_{\ast}\exp(tE_{\ast}))-f_m(x_{\ast}))|\nonumber \\
&\qquad \leq |f(x\exp(tE_{\ast})-f(x))-(f(x_{\ast}\exp(tE_{\ast})-f(x_{\ast}))| \nonumber \\
&\qquad \leq K(Ef(x)-E_{\ast}f(x_{\ast}))^{\frac{1}{2s^2}}|t|.
\end{align}
Since $t$ is small, by \eqref{estimated2} we get
\begin{align*}
&|(f_m(x_{\ast}\exp(tE_{\ast})) - f_m(x_{\ast})) -(f_m(x_{m-1}\exp(tE_{m-1}))-f_m(x_{m-1}))|\\
&\qquad \leq (E_{\ast}f_m(x_{\ast})-E_{m-1}f_m(x_{m-1})+\sigma_{m-1})|t|.
\end{align*}
Lemma \ref{lemmaquasifinale} implies that the third term of \eqref{toestimate} is bounded above by $\sigma_{m-1}|t|$.
By combining the estimates of each term and using \eqref{stima32} we get
\begin{align}\label{hi}
&|(f_m(x\exp(tE_{m-1}))-f_m(x)) - (f_m(x_{m-1}\exp(tE_{m-1}))-f_m(x_{m-1}))| \nonumber \\
&\qquad \leq (K(Ef(x)-E_{\ast}f(x_{\ast}))^{\frac{1}{2s^2}} + E_{\ast}f_m(x_{\ast})-E_{m-1}f_m(x_{m-1})+2\sigma_{m-1})|t| \nonumber \\
&\qquad \leq (K((Ef_{m}(x)-E_{m-1}f_{m}(x_{m-1}))^{\frac{1}{2s^2}}+\sigma_{m-1}/s^2) + 2\sigma_{m-1})|t|\nonumber \\
&\qquad \leq K(\sigma_{m-1}-\varepsilon'_m/2+(Ef_m(x)-E_{m-1}f_m(x_{m-1}))^{\frac{1}{2s^2}})|t|,
\end{align}
where we have used $\varepsilon'_m\leq \sigma_{m-1}/2$ and $K\geq 4s^2$ in the final line. This gives the required estimate of \eqref{incases} for small $t$.

\medskip

\emph{Suppose $3C_D\delta_{\varepsilon}^{\frac{1}{s}}/\varepsilon_{m}' \leq |t|\leq 1$.}
To estimate \eqref{incases} we use the inequality
\begin{align*}
&|(f_m(x\exp(tE_{m-1}))-f_m(x))-(f_m(x_{m-1}\exp(tE_{m-1}))-f_m(x_{m-1}))|\\
&\qquad \leq |(f_m(x_{\ast}\exp(tE_{m-1}))-f_m(x_{\ast}))\\
&\qquad \quad \quad -(f_m(x_{m-1}\exp(tE_{m-1}))-f_m(x_{m-1}))|\\
&\qquad \quad +|f_m(x_{\ast}) - f_m(x)| + |f_m(x\exp(tE_{m-1})) - f_m(x_{\ast}\exp(tE_{m-1}))|.
\end{align*}
Lemma \ref{lemmaquasifinale}(5) shows that the first term on the right hand side is bounded above by
\[K(\sigma_{m-1}-\varepsilon'_m+( E_{\ast}f_m(x_{\ast})-E_{m-1}f_m(x_{m-1}))^{\frac{1}{2s^2}})|t|.\]
The second term is bounded by $d(x_{\ast},x)\leq \delta_{\varepsilon} \leq \varepsilon_{m}'|t|\leq K\varepsilon_{m}'|t|/s^2$. For the third term, we use Lemma \ref{conjugatedistance} with $x=\exp(t E_{m-1})$ and $y=x_{\ast}^{-1}x$ to get 
\begin{align*}
&|f_m(x\exp(tE_{m-1})) - f_m(x_{\ast}\exp(tE_{m-1}))|\\
&\qquad \leq d(x\exp(tE_{m-1}), x_{\ast}\exp(tE_{m-1}))\\
&\qquad \leq C_D\left( d(x_{\ast}, x) + t^{\frac{1}{s}}d(x_{\ast}, x)^{\frac{s-1}{s}}+ t^{\frac{s-1}{s}}d(x_{\ast}, x)^{\frac{1}{s}}\right)\\
&\qquad \leq C_D(\delta_m+\delta_m^{\frac{s-1}{s}}+\delta_m^{\frac{1}{s}})\\
&\qquad \leq 3C_D \delta_m^{\frac{1}{s}}\\
&\qquad \leq \varepsilon_{m}'|t|\\
&\qquad \leq K\varepsilon_{m}' |t|/(4s^2).
\end{align*}
Combining the three estimates and using $E_{\ast}f_m(x_{\ast})\leq Ef_m(x)$ gives
\begin{align*}
&|(f_m(x\exp(tE_{m-1}))-f_m(x))-(f_m(x_{m-1}\exp(tE_{m-1}))-f_m(x_{m-1}))| \\
&\qquad \leq K(\sigma_{m-1}-\varepsilon'_m/2+( Ef_m(x)-E_{m-1}f_m(x_{m-1}))^{\frac{1}{2s^2}})|t|.
\end{align*}
This gives the required estimate of \eqref{incases} for large $t$ and therefore
\[(x_{m-1},E_{m-1})\leq_{(f_m,\sigma_{m-1}-\varepsilon'_m/2)}(x,E),\]
which concludes the proof.
\end{proof}

We conclude this section proving Proposition \ref{DoreMaleva} and  Theorem \ref{maintheorem}.

\begin{proof}[Proof of Proposition \ref{DoreMaleva}]
Proposition \ref{DoreMaleva} easily follows from Lemma \ref{lemmaquasifinale} and Lemma \ref{almostlocmax}. Indeed, Lemma \ref{lemmaquasifinale} states that there is $f\colon \mathbb{G} \to \mathbb{R}$ Lipschitz such that $f-f_{0}$ is linear and $\mathrm{Lip}_{\mathbb{G}}(f-f_{0})\leq 2t_{0}\leq \mu$. It also states there is $(x_{\ast}, E_{\ast})\in D^{f}$ satisfying $d(x_{\ast},x_{0})<\delta_{0}$ and $E_{\ast}f(x_{\ast})>0$. Lemma \ref{almostlocmax} then shows that $E_{\ast}f(x_{\ast})$ is almost locally maximal in the sense of Proposition \ref{DoreMaleva}.
\end{proof}

\begin{proof}[Proof of Theorem \ref{maintheorem}]
Let $B\subset V_1$ be a ball of directions as in Assumption 6.1. Let $f_{0}\colon \mathbb{G}\to \mathbb{R}$ be a Lipschitz function. Multiplying $f_{0}$ by a non-zero constant does not change the set of points where it is Pansu differentiable. Hence we can assume without loss of generality that $\mathrm{Lip}_{\mathbb{G}}(f_0)\leq 1/4$.

Fix an arbitrary pair $(x_{0},E_{0})\in D^{f_{0}}$ and apply Proposition \ref{DoreMaleva} with $\delta_{0}=1$, $\mu=1/4$ and $K=4s^2$. This gives a Lipschitz function $f\colon \mathbb{G}\to \mathbb{R}$ such that $f-f_{0}$ is $\mathbb{G}$-linear with $\mathrm{Lip}_{\mathbb{G}}(f-f_{0})\leq 1/4$ and a pair $(x_{\ast},E_{\ast})\in D^{f}$ with $x_{\ast}\in N$ and $E_{\ast}f(x_{\ast})>0$ which is almost locally maximal in the following sense.

For any $\varepsilon>0$ there is $\delta_{\varepsilon}>0$ such that whenever $(x,E)\in D^{f}$ satisfies both
\begin{enumerate}
\item $d(x,x_{\ast})\leq \delta_{\varepsilon}$,  $Ef(x)\geq E_{\ast}f(x_{\ast})$, and
\item for any $t\in (-1,1)$:
\begin{align*}
&|(f(x\exp(tE_{\ast}))-f(x))-(f(x_{\ast}\exp(tE_{\ast}))-f(x_{\ast}))|\\
& \qquad \leq 4s^2|t| ( Ef(x)-E_{\ast}f(x_{\ast}) )^{\frac{1}{2s^2}},
\end{align*}
\end{enumerate}
then
\[Ef(x)<E_{\ast}f(x_{\ast})+\varepsilon.\]

Since $\mathrm{Lip}_{\mathbb{G}}(f_{0})\leq 1/4$ and $\mathrm{Lip}_{\mathbb{G}}(f-f_{0})\leq 1/4$ we have $\mathrm{Lip}_{\mathbb{G}}(f)\leq 1/2$. Notice that $(x_{\ast},E_{\ast})$ is also almost locally maximal in the sense of Theorem \ref{almostmaximalityimpliesdifferentiability}, since the restriction on pairs above is weaker than that in Theorem \ref{almostmaximalityimpliesdifferentiability}. Hence Theorem \ref{almostmaximalityimpliesdifferentiability} implies that $f$ is Pansu differentiable at $x_{\ast}\in N$. Since a $\mathbb{G}$-linear function is Pansu differentiable everywhere, it follows $f_{0}$ is Pansu differentiable at $x_{\ast}$. This proves Theorem \ref{maintheorem}.
\end{proof}

\end{document}